\documentclass[12pt, draft]{amsart}
\usepackage{amssymb,amsbsy,amsmath,amsfonts,times}
\usepackage{latexsym,euscript,exscale}
\usepackage{mathrsfs}

\usepackage[all]{xy}

\usepackage{helvet}

  \newcounter{itemizedlistcounter}  % number the items
\newenvironment{itemized}
  {\begin{list}
     {(\arabic{itemizedlistcounter})} % labeling 
     {\usecounter{itemizedlistcounter}   % set counter
      \setlength{\leftmargin}{2.1em}} % set spacing 
  }
  {\end{list}}

\author[F. Cabello S\'anchez]{F\'elix Cabello S\'anchez}
\address{Departamento de Matem\'aticas and IMUEx\\
Avenida de Elvas, 06071-Badajoz, Espa\~na.
\newline
Orcid Id: 0000-0003-0924-5189}
\email{fcabello@unex.es}
\author[V. Ferenczi]{Valentin Ferenczi}
\address{Instituto de Matem\'atica e Estat\'istica \\
 Universidade de S\~ao Paulo \\
rua do Mat\~ao 1010 \\
Cidade Universit\'aria \\
05508-90 S\~ao Paulo, SP \\
Brazil \\   \newline and 
\newline
Equipe d'Analyse Fonctionnelle \\
Institut de Math\'ematiques de Jussieu\\
Sorbonne Universit\'e - UPMC \\
Case 247, 4 place Jussieu \\
75252 Paris Cedex 05 \\
France.
\newline
Orcid Id: 0000-0001-5239-111X}
\email{(corresponding author) ferenczi@ime.usp.br}
\author[B. Randrianantoanina]{Beata Randrianantoanina}
\address{Department of Mathematics\\
Miami University\\
Oxford, OH 45056, USA}
\email{randrib@miamioh.edu}

\date{}
\newcommand{\norm}[1][\cdot]{\Vert #1\Vert}
\newcommand{\triple}[1]{|\hspace{-1pt}|\hspace{-1pt}|#1|\hspace{-1pt}|\hspace{-1pt}|}

\newcommand{\eps}{\varepsilon}

\newcommand{\iso}{\rm Isom}

\newcommand{\ro}{\rho}

\newcommand {\F}{\mathbb F}

\newcommand{\lop}{\curvearrowright}

\newcommand {\N}{\mathbb N}
\newcommand {\Q}{\mathbb Q}

\newcommand {\Z}{\mathbb Z}

\renewcommand{\leq}{\ensuremath{\leqslant}}
\renewcommand{\geq}{\ensuremath{\geqslant}}

\renewcommand{\dim}{\ensuremath{\mathop{\rm dim\,}}}

\newcommand{\supp}{\operatorname{supp}}

\newcommand{\mbb}[1]{\mathbb{#1}}
\newcommand{\mr}[1]{\mathrm{#1}}

\newcommand{\de}{\delta}
\newcommand{\vep}{\varepsilon}
\newcommand{\ga}{\gamma}

\DeclareMathOperator{\age}{\mathrm{Age}}

\newcommand{\acts}{\curvearrowright}

\newcommand{\Span}{\operatorname{span}}
\providecommand{\FR}{\mathop{\rm FR}\nolimits}
\providecommand{\GL}{\mathop{\rm GL}\nolimits}

\providecommand{\supp}{\mathop{\rm supp}\nolimits}

\providecommand{\Isom}{\mathop{\rm Isom}\nolimits}

\newcommand{\Ker}{\operatorname{Ker}}

\newcommand{\al}{\alpha}

\newcommand{\g}{\gamma}
\newcommand{\ep}{\varepsilon}

\newcommand{\s}{\sigma}
\newcommand{\G}{\Gamma}

\newcommand{\bbN}{\mathbb{N}}

\newcommand{\lb}{\label}

\newcommand{\DEF}{\buildrel {\mbox{\tiny def}}\over =}
\newcommand{\til}{\rightarrow}

\newcommand{\Emb}{\rm Emb}
\newcommand{\ultrafilter}{{\mathcal U}}

\newcommand {\ku} {\mathcal}

\newcommand{\auh}{AUH}

\newtheorem{thm}{Theorem}[section]
\newtheorem{cor}[thm]{Corollary}

\newtheorem{theorem}[thm]{Theorem}
\newtheorem{theo}[thm]{Theorem}
\newtheorem{lemme}[thm]{Lemma}
\newtheorem{lemm}[thm]{Lemma}
\newtheorem{lemma}[thm]{Lemma}
\newtheorem{prop} [thm] {Proposition}

\newtheorem{fact}[thm]{Fact}

\newtheorem{ex}[thm]{Example}

\theoremstyle{definition}
\newtheorem{defi} [thm] {Definition}
\newtheorem{defin} [thm] {Definition}
\newtheorem{definition} [thm] {Definition}

\newtheorem{rem}[thm]{Remark}

\newtheorem{prob}[thm]{Problem}
\newtheorem{problemo}[thm]{}

\newtheorem{quest}[thm]{Question}

\title{On Mazur rotations problem and its multidimensional versions}



\thanks{
The research of the first author was supported in part by Projects MICINN MTM2016-76958-
C2-1-P and PID2019-103961GB-C21 and Project IB16056, Junta de Extremadura.
The research of the second author was supported by FAPESP, grants 2016/25574-8 and by CNPq, grant 303731/2019-2. This survey will be part of the special issue of the S\~ao Paulo Journal of Mathematical Sciences dedicated to the Golden Jubilee of the Institute of Mathematics and Statistics of the University of S\~ao Paulo.}

\begin{document}

\begin{abstract}The article is a survey related to a classical 
unsolved problem in Banach space theory, appearing in Banach's famous book in 1932, and known as the Mazur rotations problem. Although the problem seems very difficult and rather abstract, its study sheds
new light on the importance of norm symmetries of a Banach space, demonstrating sometimes unexpected connections with  renorming theory and differentiability in functional analysis, with  topological group theory and the theory of representations, with the area of amenability, with Fra\"iss\'e theory and Ramsey theory, and led to development of concepts of  interest independent of Mazur problem.
This survey focuses on results that have been published after 2000, stressing
two lines of research which were developed in the last ten years. The first one is the study of approximate versions of Mazur rotations problem in its various aspects, most specifically in the case of the Lebesgue 
%Banach
spaces $L_p$. The second one concerns recent developments of multidimensional formulations of Mazur rotations problem and associated results. Some new results %\m{B: removed: `and observations'} 
are also included.
\end{abstract}
\maketitle
\tableofcontents

\section{Introduction and first results on Mazur problem}\label{sec:intro}

\subsection{Mazur rotations problem} 
Hilbert spaces have the following rotations 
property: 
\begin{itemized}
\item[~]
Given two points $x,y$ on the unit sphere there exists an isometry $T$ (defined on the whole space) such that $y=Tx$. 
\end{itemized}
Here and throughout the paper {\em isometry} means {\em linear surjective isometry}. This clearly follows from the existence of orthogonal complements and can be rephrased by saying that the isometry group acts transitively on the unit sphere.

Mazur problem, which can be found in Banach's  {\em Th\'eorie des Op\'era\-tions Lin\'eaires}, asks whether every separable Banach space enjoying the above {\em rotations} property has to be isometric to {\em the} separable Hilbert space; see \cite[la remarque \`a la section~5 du chapitre XI]{Banach}. This question is perhaps best understood as two separate problems, both of which remain open to this day.

\begin{prob}[Mazur rotations problem, the isomorphic part]
Assume $X$ is a separable Banach space whose isometry group acts transitively on its unit sphere. Is $X$ linearly isomorphic to the separable Hilbert space $\ku H$?
\end{prob}

As we shall see very soon both the separability and the completeness conditions are necessary since otherwise there are easy counterexamples based on the Lebesgue spaces $L_p$. The other part of the problem, where neither completeness or separability seems to be essential, reads as follows:

%We shall detail later  on that 
 %for the problem to be non-trivial, the separability condition in the hypothesis is necessary. For example the isometry group of $L_p$ induces a dense orbit on the sphere  (see \cite{rolewicz,GJK}) and thus the isometry group will act transitively on any ultrapower of $L_p$, which itself is an $L_p$-space and therefore not a Hilbert space.

\begin{prob}[Mazur rotations problem, the isometric part]
Assume $\|\cdot\|$ is an equivalent norm on a Hibert space $\mathcal H$ whose isometry group acts transtitively on the unit sphere. Is $\|\cdot\|$ necessarily euclidean, that is, induced by an inner product on $\mathcal H$?
\end{prob}

\subsection{Notation, conventions}\label{sec:not+con}
We tend to use {$X, Y, Z, U\dots $} for infinite dimensional Banach spaces and { $A, B, E, F\dots$}  for  finite dimensional ones.
The unit sphere of $X$ is the set $S_X=\{x\in X:\|x\|=1\}$.

The space of operators from $X$ to $Y$ is denoted by $\mathcal{L}(X,Y)$. Operators are invariably assumed linear and continuous.  The identity operator on $X$ is denoted by ${\bf I}_X$.
We use
{$\operatorname{GL}(X)$} for the group of linear automorphisms of $X$
and 
{${\rm Isom}(X)$} for its group of isometries. Recall that throughout, isometries are assumed to be linear and surjective (it is worth recalling here that by Mazur's theorem, onto isometries fixing $0$ are necessarily linear). An operator $T: X\to Y$  which preserves the norm ($\|Tx\|=\|x\|$ for all $x\in X$) is called an isometric embedding and we denote the  subset of such operators in  $\mathcal{L}(X,Y)$ by $\operatorname{Emb}(X,Y)$. An operator $T:X\to Y$, not necessarily surjective, that satisfies the estimate $(1+\eps)^{-1}\|x\|\leq \|Tx\|\leq (1+\eps)\|x\|$ is called an $\eps$-isometry. We denote by $\operatorname{Emb}_{\eps}(X,Y)$ set of all $\eps$-isometries from $X$ to $Y$. 

The (multiplicative) Banach-Mazur distance between two Banach spaces $X,Y$ is defined by $$d_{{\rm BM}}(X,Y)=\inf\big\{\|T\|\|T^{-1}\|: T \text{ is an isomorphism between $X$ and $Y$}\big\},$$ 
where the infimum of the empty set is treated as $\infty$.

If $G$ is a group acting on a set $X$, meaning that we have a homomorphism $\pi$ from $G$ to the group of bijections of $X$, then the orbit of $x$ under the action of $G$ is the set $\{\pi(g)(x):g\in G\}$.  If no confusion can arise, we often identify $g$ with $\pi(g)$ and we use the notation $G\curvearrowright X$ to indicate that $G$ acts on $X$. If both $G$ and $X$ carry topologies we say that an action $G\curvearrowright X$ is continuous if the obvious map $G\times X\to X$ sending $(g,x)$ to $\pi(g)(x)$ is continuous. 
A topological group $G$ is
said to be amenable if every continuous {\em affine} action of $G$ on a compact {\em convex} set {\em of
a locally convex space} has a fixed point. By deleting all the words set in  italics one obtains the
notion of an extremely amenable group.

General references about classical but sometimes specific concepts in Banach space theory (convexity, type, cotype, asymptotic structure, finite representability, Orlicz spaces, Tsirelson space, etc...) are, for example,  \cite{LT, FHHSPZ, AK} or the chapter by Johnson and Lindenstrauss  \cite{JL-hand} opening {\em the} Handbook. 
%\m{Hope this suffices}

%MORE REFerences?\m{needs attention: Handbook? other books?, maybe Albiac Kalton?
%V:I agree certainly on including Albiac Kalton. Not bad to put the Handbook either. What about also the book "Banach Space Theory: The Basis for Linear and Nonlinear Analysis Marián Fabian , Petr Habala , 
%Petr Hájek , Vicente Montesinos , Václavh Zizler"... or not? I am in doubt.
%F: Perhaps the chapter "basic concepts" by Johnson\&Lind is more complete and updated. I don't think that Kalton-Albiac covers everything we need, but this is clearly a good choice...}

\subsection{Topologies}
Two topologies will be relevant for us on the spaces of operators ${\mathcal L}(X,Y)$, namely the norm topology, and the strong operator topology (SOT, the topology of pointwise convergence on $X$). Their restrictions provide topologies on $\operatorname{GL}(X), \operatorname{Isom}(X)$ and $\operatorname{Emb}(E,X)$. We recall some well-known useful facts:

\begin{fact}
Both $\operatorname{GL}(X)$ and  ${\rm Isom}(X)$ are topological groups in the norm topology.
\end{fact}

The norm topology is somehow too strong to be used on isometry groups and actually it has a strong tendency to discretize them (see Comment 3 in  Section~\ref{sec:smAT} for examples of this on Lebesgue spaces).

 %This is indeed the case in the real $L_p(\mu)$ spaces over arbitrary measures unless $p=2$ and, therefore, in real Lindenstrauss spaces, in particular in spaces of AUD. See Proposition XXX in Section 2.

In general the SOT is not a group topology on $\operatorname{GL}(X)$, but things get better if one looks at bounded subgroups. In particular:

\begin{fact}
The SOT makes ${\rm Isom}(X)$ into a topological group which is Polish (separable and completely metrizable) when $X$ is separable.
\end{fact}

(We refer the reader to \cite[Chapter I, \S\,9]{kechris} for an introduction to Polish groups as well as the basic examples.)
%\m{Reference for les groupes polonais...}
These facts compel us to consider the isometry groups in the SOT topology unless otherwise stated.

We shall usually equip ${\rm Emb}_\varepsilon(F,X)$ and in particular ${\rm Emb}(F,X)$ with the distance induced by the norm on ${\mathcal L}(F,X)$. But note that here the SOT and the norm topology are equivalent when $F$ is finite-dimensional.

\subsection{Transitivity and its relatives}\label{subtrans}

A Banach space is {\it almost transitive} (AT) if given $x,y\in X$ with $\|x\|=\|y\|=1$ and $\eps>0$ there exists a surjective isometry $T$ of $X$ such that $\|y-Tx\|\leq\eps$. If this can be achieved for $\eps=0$ we say that $X$ is {\it transitive}.

Establishing a vocabulary to study these properties, Pe\l czy\'nski and Role\-wicz \cite{PR} (see also Rolewicz's %\m{added 's to rolewicz, is that right?} 
book \cite[Chapter~9]{rolewicz}) defined the norm $\|\cdot\|$ of a Banach space $X$ to be {\em maximal} if
no equivalent norm can give a strictly larger group of isometries.
If, in addition, every equivalent norm with the same isometry group as $\|\cdot\|$ is a multiple of $\|\cdot\|$ the norm is called {\em uniquely maximal}. This happens if and only if $\|\cdot\|$ is {\em convex transitive}, namely for every norm one $x$ the closed convex hull of the orbit of $x$ under the action of the isometry group is the unit ball.
One has the implications
$$
\boxed{\text{\small{Hilbert}}} \Rightarrow  \boxed{\text{\small{Transitive}}} \Rightarrow \boxed{\text{\small{AT}}} \Rightarrow \boxed{\text{\small{Convex transitive}}}\Rightarrow \boxed{\text{\small{Maximal}}}
$$ 
Most of what was known on Mazur problem and its more or less natural variations until the year 2000 can be seen in the survey papers \cite{CS, becerra}. Here
we only recall that every Banach space is isometric to a 1-complemented subspace of an AT space; see Lusky \cite{lusky-rot} for the separable case and \cite[Theorem 2.14]{becerra} for the general case and some consequences. Thus, (almost) transitivity alone does not imply any Banach space property that passes to complemented subspaces (for example there exist transitive spaces without the Approximation Property, see \cite[Section~2]{becerra}. %\m{reference for AP added}).
We
shall however focus on more natural and important examples, some of which have stronger properties than AT or transitivity.

\subsection{Classical isometry groups and  examples of AT spaces}\label{sec:classical-isom}
We now present the examples %\m{B: changed the title a little, I also want to remove `main'} 
upholding the paper focusing primarily on AT spaces. Some of them will be revisited in Section \ref{4} in the multidimensional setting.
As we shall see, there is a wide variety of AT spaces arising in very different contexts. With the sole possible exception of Hilbert spaces, which may be seen from so many different points of view, these spaces are ``large''  in some sense which is difficult to make precise. Actually it is not easy to distinguish the spaces that can be given an equivalent AT norm from those that cannot; see Section~\ref{2} for more explanations and the basics on maximal norms.
General references for the isometries of classical function spaces and many related topics  are \cite{f-jVol1, FJ2, lacey}.

\subsubsection*{Hilbert spaces}
 If $\mathcal H$ is a Hilbert space,  then ${\rm Isom}(\mathcal H)$ is the unitary group. 
It acts transitively on the unit sphere.
Moreover, if $x, y$ are normalized then there is an isometry $T$ sending $x$ to $y$ such that $\|T-{\bf I}_\mathcal H\|=\|y-x\|$ (optimal) with $T-{\bf I}_\mathcal H$ of rank 2. There is another isometry $L$ sending $x$ to $y$ with  $L-{\bf I}_\mathcal H$ of rank 1 (optimal), but $\|L-{\bf I}_\mathcal H\|=2$.
%, meaning there is a single and full orbit for the action
%${\rm Isom}(H) \lop S_H$.  

\subsubsection*{Lebesgue spaces}
Given a measure $\mu$ defined on a set $S$ and $1\le p<\infty$ we denote by $L_p(\mu)$ the usual Lebesgue space of $p$-integrable functions on $S$, with the usual convention about identifying functions that agree almost everywhere. If $\mu=\lambda$ is the Lebesgue measure on the unit interval we just write $L_p$.
If $\phi:[0,1]\to[0,1]$ is a Borel automorphism (a bijection preserving Borel sets in both directions) which preserves null sets in both directions and $h$ is a measurable function such that $|h|^p=d(\lambda\circ\phi)/d\lambda)$, that is $\lambda(\phi(B))=\int_B|h|^pd\lambda$ for every Borel $B\subset[0,1]$, then the operator 
$$
(Tf)(t)=h(t)f(\phi(t))
$$
is a correctly defined isometry of $L_p$. If $p\neq 2$ the converse is also true and every $T\in\operatorname{Isom}(L_p)$ arises in this way (the Banach-Lamperti theorem \cite[Theorem~3.2.5]{f-jVol1}, although in this formulation we need a little help from von Neumann \cite{vN32}). This has the following consequences for finite $p\neq 2$:
\begin{itemize}
\item $L_p$ is AT, but not transitive: there are exactly two (dense) orbits on the unit sphere namely, the ``full support'' one, i.e. the orbit $\{f\in S_{L_p}: \lambda(f^{-1}(0))=0\}$ and the complement $\{f\in S_{L_p}: \lambda(f^{-1}(0))>0\}$.

\item The dense subspace $L_p(0,1^-)=\bigcup_{b<1}\{f\in L_p: \supp(f)\subset [0,b]\}$ is a {\em transitive} normed space (with the obvious definition).

\item If $\aleph$ is an uncountable cardinal, then the Banach space $\ell_p(\aleph, L_p)$ (which can be regarded as $L_p(\mu)$, where $\mu$ is ``Lebesgue measure'' on  $\aleph$-many disjoint copies of the unit interval) is transitive. Note that this space has density character $\aleph$, while nontrivial ultraproducts (see Section~\ref{ult} below) have density character at least the continuum.
\end{itemize}

The case $p=\infty$ was excluded in the preceding discussion because the space $L_\infty$, being a $C(K)$ in  disguise, cannot be AT. The isometries of $C(K)$ are described by the Banach-Stone theorem (1932): all have the form $Tf(x)=u(x)f(\phi(x))$, where $\phi$ is a homeomorphism of $K$ and $u:K\to\mathbb K$ is  continuous and unimodular. In particular the orbit of the unit $1$ cannot be dense in the sphere unless $K$ is a singleton. And what happens with other spaces of type $\mathcal L_\infty$? Keep reading.

\subsubsection*{The Gurariy space} 
A Banach space $U$ is said to be of {\it almost universal disposition} (AUD) if, given a finite dimensional space $F$, isometric embeddings $v:E\to F, u:E\to U$ and $\eps>0$, there exists an $\eps$-isometry $w:F\to U$ such that $u=wv$. Diagramatically,
$$
\xymatrixrowsep{1.25pc}
\xymatrix{
E \ar[rd]_u\ar[rr]^v && F \ar[dl]^w \\
 & U
}
$$
This notion was coined by Gurariy in \cite{gurariy}, where he constructed the space that bears his name as a separable space of AUD. Gurariy also established that two separable Banach spaces of AUD are ``almost isometric''
(that is, the Banach-Mazur distance between them is equal to 1) and that for every $\eps>0$ the surjective $\eps$-isometries act transitively on the unit sphere of any separable space of AUD. Although this is not completely evident from the definition, any space of AUD must be a Lindenstrauss space (i.e. a predual of an $L_1$-space) %\m{added parenthesis remark}
because AUD implies the following extension property of $X$: given a subspace $E$ of a finite dimensional space $F$ and $\eps>0$ every operator $\tau:E\to X$ has an extension $\tilde{\tau}:F\to X$ with $\|\tilde{\tau}\|\leq (1+\eps)\|\tau\|$.

The isometric uniqueness of the Gurariy space   $\mathcal{G}$ was finally established by Lusky in a fine paper \cite{luskygurarij} where he also showed that the isometry group acts transitively on the set of smooth points of the sphere of $\mathcal{G}$. See \cite[Section~4]{f-w} for more general results concerning finite-dimensional subspaces of  $\mathcal{G}$.

A new proof of the uniqueness of the Gurariy space was later provided by Kubi\'s and Solecki in \cite{kubissolecki}: they basically proved that the Gurariy space is the (approximate) Fra\"iss\'e limit of the class of finite dimensional Banach spaces and isometric embeddings. This remarkable feature of the Gurariy space inspired  the study of the interactions between Fra\"\i ss\'e structures and Banach spaces; see \cite{Lup} and the references therein.
We shall pursue this approach in
Section~\ref{4}. From another point of view, see also the recent description by C\'uth, Dole\u{z}al, Doucha and Kurka, of the Gurariy space as the ``generic" separable space \cite{CDDK}.

\subsubsection*{The Garbuli\'{n}ska space} One should speak, more accurately, of the Garbuli\'nska-W\c egrzyn renorming of the Kadec/Pe\l czy\'nski/Wojtaszczyk space, see below.
The Garbuli\'{n}ska space plays the same role as the Gurariy space in a different category, where one takes into account 1-complemented subspaces only. Let us say that  a Banach space $X$ has the property [$\Game$] if given isometries with 1-complemented range $u:E\to X$ and $v:E\to F$, where $F$ is finite-dimensional, and $\eps>0$ there is an $\eps$-isometry $w:F\to X$ with $(1+\eps)$-complemented range such that $u=wv$.

Garbuli\'nska shows in \cite{garb} that there exists a unique, up to isometries, Banach space $\mathcal K$ with a skeleton and  property [$\Game$]. Recall that a skeleton of $X$ is a chain of finite dimensional subpaces $(E_n)_{n\geq 1}$ whose union is dense in $X$ and such that $E_n$ is 1-complemented in $E_{n+1}$.
This condition is a clear analogue of separability in the 1-complemented category and is just a transcription
 of 1-FDD.
Most isometric properties of $\mathcal K$ depend, one way or another, on the following fact (\cite[Theorem 6.3]{garb}):
%\smallskip
 
 \begin{itemized}
\item[~]Let $K$ and $K'$ be Banach spaces with skeletons, satisfying the property 
$[\Game]$, and let $h : A \to B$  be an isometry between 
1-complemented finite-dimen\-sional subspaces of $K$ and $K'$, respectively. Then for
every $\eps>0$ there exists an isometry $H:K\to K'$ such that $\|H(x)-h(x)\|\leq \eps\|x\|$ for all $x\in A$. In particular, $K$ and $K'$ are isometric.
\end{itemized}
%\smallskip

Since all lines in a Banach space are 1-complemented and isometric to each other, it follows that $\mathcal K$ is AT. 

Another important feature of $\mathcal K$, that is going to play its role in Section~2,
%\m{correct section number F: fixed}
is that $\mathcal K$ contains a 1-complemented copy of every space with a skeleton. This makes $\mathcal{K}$ isomorphic to some old acquaintances in the theory of complementably universal spaces. Let $\mathcal{C}$ be a class of Banach spaces. We say that a Banach space is (complementably) universal for $\mathcal C$ if it belongs to $\mathcal C$ and it contains a (complemented) isomorph of each space in $\mathcal C$.
This concept emerged in the paper \cite{P69}, where  Pe\l czy\'nski constructed his celebrated (space with a) universal basis (call it $P_B$)
%\m{B: I changed the name from P to $P_B$ to avoid confusion with the notation for the pseudo-arc which is immediately below, and which seems to be used more throughout the paper}
which is a complementably universal space for the class of Banach spaces with bases and a similar space with an unconditional basis which we shall denote by $U$.

Later on M.\u{I}. Kadec \cite{K71} exhibited a complementably universal space for the bounded approximation property (BAP); let us denote that specimen by $K$ and observe that an obvious application of the Pe\l czy\'nski decomposition method shows that any two 
complementably universal spaces for the BAP are isomorphic. In the same issue of {\it Studia} where Kadec'  space first appeared, Pe\l czy\'nski \cite{P71} showed that each separable space with the BAP is complemented in a space with a basis: the inexorable consequence is that the spaces $P_B$ and $K$ are isomorphic. But since the Garbuli\'nska space has the BAP (obvious) and each Banach space with a basis can be renormed to get a skeleton (even more obvious) we can apply again the Pe\l czy\'nski decomposition method to conclude that the Garbuli\'nska space $\mathcal{K}$ is isomorphic to $P_B$ and
 $K$, which are also isomorphic to a  space complementably  universal for FDDs  constructed by Pe\l czy\'nski and  Wojtaszczyk  \cite{PW71} in the very same volume of {\it Studia}.

\subsubsection*{Spaces of continuous functions on the pseudoarc}
Although 
regarding AT spaces of type $\mathcal L_\infty$ the Gurariy space is the guy to work with, there are other natural examples. One of them is the separable ``$M$-Gurariy'' space from \cite{FLMT} and a  closely related, but  non-separable creature is obtained in \cite{CS7} taking ultraproducts of the spaces  $L_p$ with variable $p\to\infty$; cf. Comment 4 in Section 4.9.

  Here we discuss spaces of continuous functions 
  %\m{B: change}
  based on the pseudoarc, a continuum constructed by Knaster \cite{K22} in the 1920s which became a celebrity in {\em certain} circles because of the Bing's characterization: it is the only hereditarily indecomposable chainable continuum; let us denote it by $P$.   An impressive wealth of well organized information on the pseudoarc is contained in   Lewis' survey \cite{lewis}.

Kawamura \cite{kawa} and Rambla \cite{rambla}, 
independently and almost simultaneously, proved that if $P_*$ is the pseudoarc with one point removed, then the complex space $C_0(P_*)$ is AT in the sup norm, 
thus refuting a long standing conjecture of Wood \cite[Section~3]{W}. The group of homeomorphims acts transitively on $P$ and so the homeomorphic type of $P_*$ does not depend of which point is removed.

Curiously enough the pseudoarc can be considered as the (inverse) Fra\"\i s\'e limit of a suitable class as shown by Irwin and Solecki in \cite{i-s} which is simply delighful, given the approach of this survey, cf. Section 4.

Taking ultrapowers leads to $C_0(L)$-spaces which are transitive in the sup norm.

Naive observations in \cite{c-glasgow} suggest that if $L$ is a locally compact space with more than one point, the complex space $C_0(L)$ is separable and AT, then the one-point compactification of $L$ should %\m{F: small change}
be (homeomorphic to) the pseudoarc.

As for real spaces, Greim and Rajalopagan proved in \cite{g-r} that no   $C_0(L)$,  can be AT in the sup norm if $L$ has more than one point. However, curiously enough,
 there exists a quite natural norm under which a real $C_0(L)$ can be AT and even transitive. Indeed, if $f:L\to\mathbb R$ is any function, we set  
$$\operatorname{diam}(f)=\sup\{|f(x)-f(y)|:x,y\in L\}.$$
If $L$ is locally compact but not compact, then $\operatorname{diam}$ is a norm on $C_0(L)$, clearly equivalent to the sup norm. If $K$ is compact then $\operatorname{diam}$ vanishes on the constant functions and so it defines a true norm on $C(K)/\mathbb K$ which agrees with the quotient norm (up to a factor) in case of real scalars. It is shown in \cite[Lemma~3.1]{c-convex} that $C(P)/\mathbb R$ is AT and thus the real space $C_0(P_*)$  equipped with the diameter norm is AT. It is perhaps worth noticing that both the isometry group of 
the complex space $C_0(P_*)$ and that of $C(P)/\mathbb R$ fail to be  amenable in the SOT (\cite[Example 3.2]{c-convex}).

\subsubsection*{Noncommutative $L_p$-spaces} Other families of AT spaces come from the noncommutative generalizations of $L_p$. We shall not even give the definition and we refer the reader to the {\em official} sources \cite{haa, terp} instead, but let us mention that there is a classical construction in operator algebras, due to Haagerup, that associates to each von Neumann algebra $\mathcal M$ a family of spaces $L_p(\mathcal M)$ for $p\in(0,\infty]$ in such a way that $L_1(\mathcal{M})=\mathcal{M}_*$ is the predual of $\mathcal{M}$ and $L_\infty(\mathcal{M})=\mathcal{M}$. The Haagerup $L_p(\mathcal M)$-spaces consist of certain unbounded operators acting on a Hilbert space which is related to $\mathcal{M}$ in a highly nontrivial way.

By a celebrated result of Connes and St\o rmer \cite[Theorem~4]{connes}, if $\mathcal M$ is a factor of type III$_1$, then, given states $\phi,\psi\in\mathcal M_*$ and $\varepsilon>0$, there is a unitary $u\in\mathcal M$ such that $\|u^*\phi u-\psi\|_{\mathcal M_*}<\eps$, where $u^*\phi u$ is defined by $\langle u^*\phi u,x\rangle=   \langle \phi ,uxu^*\rangle$ for $x\in \mathcal M$. 
 It follows from the generalized Power-St\o rmer inequality (see    \cite[Appendix]{hiai}) that the spaces $L_p(\mathcal{M})$ for finite $p$ have a similar homogeneity property: given positive $f,g\in L_p(\mathcal{M})$ with $\|f\|_p=\|g\|_p=1$ and $\eps>0$ there is a unitary $u\in\mathcal M$ such that $\|u^*f u-g\|_{p}<\eps$. It  follows readily that for arbitrary $f,g\in L_p(\mathcal{M})$ with $\|f\|_p=\|g\|_p=1$ and $\eps>0$ there exist unitaries $u,v\in\mathcal M$ such that $\|vf u-g\|_{p}<\eps$ and so $L_p(\mathcal{M})$ is AT.

By  remarks on ultraproducts presented in Section~\ref{ult} below, the countable ultrapowers of $L_p(\mathcal{M})$ are transitive and, by results of Raynaud \cite{ray}, the ultrapowers $L_p(\mathcal{M})_\ultrafilter$ can
be represented as the Haagerup spaces $L_p(\mathcal N)$, for some large von Neumann algebra $\mathcal N$.

\subsection{Microtransitivity}
In a desperate attempt to break the impasse on Mazur problem the authors of \cite{CDKKLM} and \cite{BRP} 
%\m{new reference added}
consider the following stronger form of transitivity, which has very little to do with the subject of this survey: a Banach space is called  microtransitive (MT) if for every $\eps>0$ there is $\delta>0$ so that if $x,y\in S_X$ satisfy $\|y-x\|<\delta$ there is $T\in\operatorname{Isom} X$ such that $y=Tx$ {\em and} $\|T-{\bf I}_X\|<\eps$.
As one may guess the only known examples of MT spaces are the Hilbert spaces, which satisfy the definition with $\delta=\eps$. 
 The issue of separability (and completeness), which is central in Mazur rotations problem, is irrelevant  for MT: $X$ is MT if and only if for every separable $Y\subset X$ there exist a further separable $Z\subset X$ which is MT and contains $Y$. Moreover, MT passes to the dual and implies both uniform convexity and uniform smoothness of the norm; see \cite[Theorems~3.11 and 3.14]{BRP} for the strongest available results in this line. %\m{sentence added}
 So, the following is a seemingly cheap, but still open, substitute for the Mazur  
problem:

\begin{prob}
Are the Hilbert spaces the only microtransitive Banach spa\-ces?
\end{prob}

\begin{small}
\noindent{\em Comments:} 
\begin{itemized}
\item[] The Effros Microtransitivity Theorem  \cite[Theorem~2.1]{effros} states that a Polish group acting transitively 
on a Polish space must act microtransitively. See also van Mill's work on this topic  \cite{vanMill}. This implies that if $X$ is a separable transitive Banach space, then the action of the isometry group on the sphere is {\em SOT}-microtransitive: i.e. for any $x \in S_X$, the map assigning to an isometry $T$ its value in $x$ is open for the SOT on $\operatorname{Isom}(X)$. The notion of microtransitivity (MT) defined above is much stronger and corresponds to the map being open in the norm topology on $\operatorname{Isom}(X)$. The papers  \cite{CDKKLM, BRP} also consider ``open actions'' of the semigroup of contractive automorphisms.%\m{sentence added}
\end{itemized}
\end{small}

\subsection{Strict convexity and transitivity}\label{sec:strict}

Though much information has been obtained on  almost transitive Banach spaces under additional geometric assumptions such as reflexivity \cite{CS, becerra}, very few conditions that are necessary for the actual transitivity
in the separable case are known.

Related to the present study, let us mention that, if $X$ is a separable   transitive real Banach space, then $X$ is strictly convex
and smooth, and thus $X^*$ is AT; see \cite[Theorem 28]{FR2} and \cite[Corollary 2.9]{becerra}.
This result fails if $X$ is only assumed to be almost transitive (resp. if $X$ is non-separable), as can be seen by considering $L_1$ (resp. an ultrapower of $L_1$, see the next section on ultraproducts). The paper \cite{aiz-pach} contains some related results. %\m{New reference added}

\subsection{Ultraproducts}\lb{ult}
The Banach space ultraproduct construction is a quite useful technique that allows  one to construct large spaces with upgraded transitivity properties. 
We refer the reader to \cite{hein} (or Sims' booklet \cite{sims}) for two very readable expositions which suffice for our modest purposes. A more complete one, which emphasizes the model-theoretic pedigree of the ultraproduct construction is \cite{h-i}. Here we only recall the definition, just to fix the notation.   

Let $(X_i)$ be a family of Banach spaces indexed by $I$ and let 
$\ultrafilter$
be an ultrafilter on $I$.

Consider the space of bounded families $\ell_\infty(I,X_i)$ equipped with the sup norm and the closed subspace $c_0^\ultrafilter(X_i)=\{(x_i): \lim_\ultrafilter\|x_i\|=0\}$. The Banach space  
 $\ell_\infty(I,X_i)/c_0^\ultrafilter(X_i)$, with the quotient norm, is called the ultraproduct of the family $(X_i)_{i\in I}$ along $\ultrafilter$ and it is denoted by $[X_i]_\ultrafilter$. When all $X_i=X$ for some fixed $X$ the ultraproducts are called ultrapowers and are denoted by $X_\ultrafilter$ instead. 
 
An ultrafilter is called free if it contains no finite set; otherwise there is exactly one point $i\in I$ such that $U\in \mathcal U\iff i\in U$ and $ \mathcal U$ is called principal.
An ultrafilter $\mathcal U$ is said to be countably incomplete (CI, for short) if there exists a countable family of members of $\mathcal U$ whose intersection does not belong to $\mathcal U$; we can require the intersection to be empty without altering the definition. 
It is very easy to see that all free ultrafilters on a countable set are CI and that $\mathcal{U}$ is CI if and only if there is a strictly positive function $f:I\to (0,1)$ such that $f(i)\to 0$ along $\mathcal{U}$. Ultraproducts are relevant in our business because of the following observation (see \cite[Proposition 2.19]{becerra} for this formulation and \cite[Remark on p. 479]{GJK} or \cite[Lemma 1.4]{CS7} for two slightly weaker forerunners):

\begin{fact}\lb{1.6}
An ultraproduct of a family of AT spaces along a CI ultrafilter is transitive.
\end{fact}

\begin{small}
\noindent{\em Comments:} 
\begin{itemized}
\item[(1)]
 It is clear that the conclusion of Fact~\ref{1.6} subsists under much weaker hypotheses. For a fixed $\eps>0$, say that $X$ is $\eps$-transitive if given $x,y\in S_X$ there is $T\in\operatorname{Isom}X$ such that $\|y-Tx\|\leq\eps$. Call it $\delta$-asymptotically transitive if, given $x,y\in S_X$ there is a surjective $\delta$-isometry $T$ such that $y=Tx$; this is inspired by Talponen's \cite[Definition 2.1]{T07}. An easy argument on series
%Question: is there something about maximality?
%Clearly ${\rm Isom}(X)_U$ maximal implies ${\rm Isom}(X)$ maximal.  The converse seems to be open.
shows that an $\eps$-transitive Banach space is also $2\eps$-asymptotically transitive provided $\eps\leq \frac12$. It is straightforward that
if $(X_n)$ is a sequence of Banach spaces such that $X_n$ is $\delta_n$-asymptotically transitive and $\delta_n\to 0$ as $n\to\infty$ and $\mathcal{U}$ is a free ultrafilter on $\mathbb N$, then the ultraproduct $[X_n]_\mathcal{U}$ is transitive. 
%\item[(2)] \m{V.: I added this comment. My concern is that someone might think we are listing "all known" transitive spaces, so I wanted to make clear there are many other (less natural) examples. Please feel free to reformulate or change the place of this comment.}The spaces considered in this section are far from listing all possible AT or transitive examples. For example and according to the classical result cited at the end of Subsection \ref{subtrans}, transitive spaces without the Approximation Property do exist.
\item[(2)]   Perhaps the most interesting question concerning transitivity properties of ultraproducts is whether the transitivity of the ultrapower $X_\mathcal U$ implies {\em anything} about the isometry group of the base space $X$.
Of course one can ask whether $X$ must be AT, which is quite natural from the point of view of model theory,  
but actually at this point it is even open whether there 
 exists  a Banach space with only trivial isometries whose ultrapowers are transitive.
\end{itemized}
\end{small}

\section{Maximality of norms, Wood's problems, Deville-Godefroy-Zizler problem}\label{2}

Recall from Section 1.4 that every transitive or even almost transitive norm is maximal. This follows easily from the observation that if a group of isomorphisms acts as an isometry for two norms, then these norms must be proportional on any orbit of the action of the group.
This led many people to investigate which spaces have maximal norms.

In 1933-34 Auerbach  \cite{A33,A33b,A34} proved that for every finite dimensional real Banach space $(X,\|\cdot\|)$, there exists a norm $\|\cdot\|_2$ on $X$ induced by an inner product and such that the isometry group of 
$(X,\|\cdot\|_2)$ contains the isometry group of $(X,\|\cdot\|)$. Thus the isometry group of every real finite dimensional space is contained in that of a maximal norm. Rolewicz \cite[\S 9.8]{rolewicz} showed that the norm of any space with a 1-symmetric basis (real or complex) is  maximal. This includes norms on the classical spaces $\ell_p$, whose isometries act as ``signed" permutations of the vectors of the unit basis - and therefore those norms are maximal but not AT. Norms of the spaces $L_p, 1 \leq p<\infty$, being AT, are in particular maximal. For $C(K)$- and specially for $C_0(L)$-spaces, the situation is more involved, depending on whether the scalars are real or complex.  See the survey paper by J.
Becerra Guerrero and \'A. Rodr\'iguez-Palacios \cite{becerra} for general information on  maximal norms and \cite{c-convex} and the references therein for maximality in $C_0(L)$ and $C(K)$-spaces.

Note that if $G$ is a bounded subgroup of $\GL(X)$, then $G$ is a subgroup of $\Isom(X,\|\cdot\|_G)$, where $\|\cdot\|_G$ is an equivalent norm  on $X$ defined by
$\|x\|_G=\sup_{g \in G}\|gx\|.$
Thus  a norm  is maximal if and only if the corresponding isometry group is a maximal bounded subgroup of $\GL(X)$. Citing the introduction of \cite{FR}, ``[it seemed] natural to suspect that a judicious choice of smoothing procedures on a space $X$ could eventually lead to a most symmetric norm, which then would be maximal on $X$".
 However the following fundamental questions  on maximal norms  remained open until 2013.

\begin{prob}[1982, Wood  \cite{W}]\lb{W1}
Does every Banach space admit an equi\-valent maximal norm, that is does $\GL(X)$ always have maximal bounded subgroups?
\end{prob}

\begin{prob}[1993, Deville, Godefroy, Zizler]
\cite[Problem~IV.2 and the remark following it]{DGZ} \lb{DGZ}
Does every super-reflexive space admit an equivalent almost transitive norm?
\end{prob}

\begin{prob}[2006, Wood \cite{Wood}] \lb{W2} 
Is it true that for every Banach space, the\-re exists an equiva\-lent maximal renorming  whose isometry group contains the original isometry group, i.e., is every bounded subgroup of $\GL(X)$  contained in a maximal   bounded subgroup of $\GL(X)$?
\end{prob}

In 2013 Ferenczi and Rosendal \cite{FR} answered these three problems negatively by  exhibiting   a complex super-reflexive space and a real reflexive space, both without a maximal bounded subgroup of the isomorphism group. 
In 2015 Dilworth and Randrianantoanina \cite{DR} studied 
Problems~\ref{DGZ} and \ref{W2} further. They showed
multiple examples of super-reflexive spaces (both complex and real) which  provide a negative answer to Problems~\ref{DGZ} and \ref{W2}, despite the fact that they have an equivalent maximal renorming.
Among others, the classical spaces $\ell_p$, $1\le p<\infty$, $p\ne 2$, are such examples. In \cite{DR}  the authors   also showed that for some spaces $X$, the group 
$\GL(X)$ may contain even continuum different maximal bounded subgroups. It is open whether there exists a Banach space $X$ with a unique, up to conjugacy, maximal bounded subgroup of  $\GL(X)$, or whether   Hilbert space has this property.%\m{Added: up to conjugacy}

%Spaces with multiple maximal bounded subgroups of 
%$\mathbf{GL(X)}$  {s:manymax}

\subsection{Almost trivial isometry groups}\lb{subatig}

In this section we describe the main result of Ferenczi and Rosendal from \cite{FR}. 
\begin{thm}\label{main1}
There exists a complex, separable, super-reflexive  Banach space $X$, and a real, separable, reflexive  space $Y$, both without  maximal bounded groups of isomorphisms,  i.e., $X$ and $Y$ have no equivalent maximal norms. 
\end{thm}

We choose to present a sketch of the result corresponding to the complex case, and to present a simplified version of the results. This allows us to give much simpler versions of the proofs of \cite{FR}. 

\

A second motivation and a source of tools for the work \cite{FR} comes from the theory of spaces with ``few operators", initiated by the construction of W.T. Gowers and B. Maurey \cite{GoMa} of a hereditarily indecomposable (or HI) space  (meaning that it contains no subspace decomposable as a direct sum of infinite dimensional subspaces).
Gowers and Maurey proved that such spaces have small spaces of operators, namely, in the complex case   any operator is a strictly singular perturbation of a scalar multiple of the identity map. The currently strongest result in this direction, due to S. A. Argyros and R. G. Haydon \cite{AH}, is the construction of a Banach space on which every operator is a compact perturbation of a scalar multiple of the identity. %Furthermore, since AH has a Schauder basis, every compact operator is a limit of operators of finite rank.

One can ask the same question for isometries. An isometry is called   trivial if its a scalar multiple of the identity. Does every Banach space admit a non-trivial surjective isometry? After partial answers by P. Semenev and A. Skorik \cite{SS}, and an answer in the real separable case by S. Bellenot \cite{Bel}, the question  was settled by K. Jarosz \cite{J}, who proved that any real or complex Banach space admits an equivalent  norm with only trivial isometries.

Thus, no isomorphic property of a space can force the existence of a non-trivial surjective linear isometry.
On the other hand it is immediate, through renormings where some prescribed finite dimensional subspace becomes euclidean, that an infinite dimensional space always admits an equivalent norm whose isometry group contains a copy of the unitary group of the $n$-dimensional euclidean space. 

In this line Ferenczi and Rosendal investigate results relating the size of the isometry group ${\rm Isom}(X,\|\cdot\|)$, for any equivalent norm $\|\cdot\|$, with the isomorphic structure of $X$, through the next definition.
%Let us first remark that any infinite-dimensional Banach space $X$ can always be equivalently renormed such that $X=F \oplus_1 H$, where $F$ is a finite-dimensional euclidean space. So, in this case, ${\rm Isom}(X)$ will at least contain a subgroup isomorphic to ${\rm Isom}(F)$. %Actually, if $X$ is a separable, infinite-dimensional, real space and $G$ is a finite group, then it %is possible to find an equivalent norm for which $\{-1,1\} \times G$ is isomorphic to the group %of isometries on $X$  \cite{stern, FG}.

%Thus, allowing for renormings, we need a less restrictive concept of when an isometry is trivial.
\begin{defi} 
A bounded  subgroup $G\leqslant \GL(X)$ acts {\em nearly trivially} on $X$ if there is a $G$-invariant decomposition  $X=F \oplus H$, where $F$ is finite-dimensional and $G$ acts by trivial isometries on $H$.
\end{defi}

The relation of this concept with questions of maximality is based on the following easy but powerful lemma:

\begin{lemma}\label{nomax} If the isometry group of an infinite dimensional space acts nearly trivially then the norm is not maximal. 
\end{lemma}

\begin{proof} If $X=F \oplus H$ is the decomposition associated to the near triviality of ${\rm Isom}(X)$, and if $H$ is decomposed as
 $R \oplus Y$, where $R$ is $1$-dimensional, then the equivalent norm defined by the formula
 $\|f\|=\|r\|+\|y\|$,  $f \in F$, $r \in R, y \in Y$, admits an isometry group which strictly contains the original one.
\end{proof}

In particular if every bounded subgroup of $\GL(X)$  (equivalently, every isometry group) acts nearly trivially, then $X$ admits no maximal renorming.

As an initial step towards Theorem \ref{main1}, Ferenczi and Rosendal, improving on some earlier work of  F. R\"abiger and W. J. Ricker \cite{RR,RR2}, show that in a certain class of spaces, each individual isometry acts nearly trivially. 

\begin{thm}\label{nearly trivial-intro}
Let $X$ be a Banach space containing no unconditional basic sequence. Then each individual isometry which is of the form $\lambda{\bf I}_X+S$, for $S$ strictly singular, acts nearly trivially on $X$ (and in particular $S$ is a finite range operator). 
In particular each isometry on a  complex HI space acts nearly trivially.
\end{thm}

\begin{proof} The spectrum of an isometry of the form ${\bf I}_X+S$ is formed either of a finite sequence of eigenvalues, or an infinite converging sequence of eigenvalues together with their limit $1$.  In the latter case the authors of \cite{FR} prove that a sequence of eigenvectors associated to eigenvalues converging fast enough to $1$ would form an unconditional basic sequence (with constant arbitrarily close to $1$). In the former case, classical spectral decomposition results imply that the operator $S$ has finite dimensional  range, or equivalently, that ${\bf I}_X+S$ acts nearly trivially on $X$. \end{proof}

For future reference the decomposition of $X$ associated to the fact that an operator $T$ acts nearly trivially may be written as $X=F_T \oplus H_T$
where $H_T$ is the kernel of ${\bf I}_X-T$ and $F_T$ its image. This notation will be used in what follows.

The next step is to proceed from single isometries acting nearly trivially to an understanding of the global structure of the isometry group ${\rm Isom}(X)$. 
 Using a renorming result of Lancien \cite{lancien} for separable reflexive $X$, the authors of \cite{FR} prove a version of Alaoglu-Birkhoff \cite{AB} ergodic decomposition theorem:
 \begin{prop} Assume $X$ is separable reflexive and $G$ is a bounded group of automorphisms of $X$. Let  $H_G$ be the subspace of points fixed by every $T \in G$, 
 and $H_{G^*}$ be the subspace of functionals fixed by every element of $G$ under its natural action on $X^*$. Let  ${\mathcal S}$ be a family generating a SOT-dense subgroup of $G$.
 
 Then $X$ admits the $G$-invariant decomposition $F_G \oplus H_G$, where 
 $$F_G=H_{G^*}^{\perp}=\overline{\rm span} \bigcup_{T \in {\mathcal S}} F_T,$$
 and the associated projection onto $H_G$ has norm at most $\|G\|^2$ (where $\|G\|:=\sup_{g \in G}\|g\|)$.
 \end{prop}
 
 Denote by ${\rm Isom}_f(X)$, the subgroup of isometries of the form
${\bf I}_X+A$, where $A$ is a finite-rank operator on $X$.
 Note that when $G$ is a subgroup of ${\rm Isom}_f(X)$, each subspace $F_T$ is finite dimensional. This leads the authors  of \cite{FR} to consider possible FDDs of $X$:
 
 \begin{prop}\label{dec} Let $X$ be separable and  reflexive. Then  
 either   ${\rm Isom}_f(X)$  acts nearly trivially on $X$, or $X$ admits a complemented subspace with a finite dimensional decomposition.
 \end{prop}
 
 \begin{proof} This is \cite[Theorem 4.16]{FR}; the proof goes as follows.
Picking an  SOT-dense sequence $(T_n)$ of isometries in ${\rm Isom}_f(X)$,  one considers each of the Alaoglu-Birkhoff decompositions associated to the subgroups 
$G_n$ generated by $T_1,\ldots,T_n$, i.e.
$$X=F_n \oplus H_n,$$ where $F_n$ is the linear span of the finite dimensional subspaces ${\rm Im}({\bf I}_X-T_j), 1 \leq j \leq  n$, and
$H_n$ the set of points fixed by $T_1,\ldots,T_n$. Consider the decomposition
$$X=F_G \oplus H_G,$$ associated to $G={\rm Isom}_f(X)$.
It can be seen that $F_G$ identifies with the closure of $\bigcup_n F_n$ and therefore either is finite dimensional or admits an FDD. In the former case $G$ acts trivially on the finite codimensional space $H_G$ and therefore nearly trivially on $X$.
\end{proof}

Combining Theorem~\ref{nearly trivial-intro} and Lemma \ref{nomax}, with the decomposition from Proposition \ref{dec}, along with the indecomposability property of HI spaces, one deduces: 

\begin{thm}\label{main2} 
Let $X$ be a separable, reflexive, hereditarily indecomposable, complex Banach space without a FDD. Then for any equivalent norm on $X$, the group of isometries acts nearly trivially on $X$. In particular $X$ does not admit a maximal norm.
\end{thm}

The existence of a uniformly convex example satisfying these conditions follows from 
an earlier construction of a super-reflexive HI space due to Ferenczi \cite{F}, as well as conditions by Szankowski \cite{Szankowski} for the existence of subspaces failing the Approximation Property and therefore failing to have an FDD.

\

\begin{small}
\noindent{\em Comments:} 
\begin{itemized}
\item 
The construction of \cite{FR}
does not seem to provide a uniformly convex space on which no subspace admits an AT norm.
Indeed on subspaces admitting a Schauder basis, the authors also obtain isometry invariant decompositions of the form $F \oplus H$ where $F$ is finite dimensional, but are only able to prove that the group of isometries which are finite range perturbations of the identity acts as an SOT-discrete group for on $H$. We are unaware of a general argument suggesting that this would prevent the existence of dense orbits for the action of the isometry group on the sphere (on this subject, one can consult \cite{AFGR}  where an SOT-discrete bounded group of automorphisms is constructed on $c_0$
whitout discrete orbits). Uniformly convex spaces where no subspaces admit an AT renorming will be encountered in Section~\ref{subDGZ}.

\vspace{1mm}

\item%[(2)]
The paper \cite{FR} has a wider scope than presented above. First of all, by renorming, bounded groups of automorphisms may be seen as groups of isometries, to which the above results apply. The setting of several results may also be extended from the case of spaces with few operators, to a more general case of bounded actions of groups of operators of the form ${\bf I}_X+S$ on arbitrary Banach spaces.  Through a finer analysis of the group structure of the isometries, FDD may be replaced by Schauder bases in most occurences. Finally methods of complexification allow to extend most results to the real case.

\vspace{1mm}\item%[3]
 Weaker   forms of rigidity than in the exotic spaces considered in this section may also induce restrictions on the actions of bounded groups. For considerations in this line regarding bounded groups acting on interpolation scales and (almost) transitivity, see Section 3.4 of the recent paper \cite{CF}.

\vspace{1mm} \item%[4] 
Before   leaving the topic of ``nearly trivial isometries'', it is perhaps worth noticing the following result from \cite{c-studia}: {\em If $\Isom_f(X)$ acts transitively on the unit sphere of a normed space $X$ then the norm of $X$ is Euclidean.}

\end{itemized}
\end{small}

\subsection{More on the Deville-Godefroy-Zizler problem}\lb{subDGZ}

In this section we describe the results and methods of Dilworth and Randrianantoanina \cite{DR} providing additional counterexamples for Problem~\ref{DGZ}.  They proved the following.

\begin{theorem}\lb{list}
The following classes of Banach spaces do not admit an equivalent almost transitive renorming.

\begin{itemize}
\item[(a)]    subspaces of classical sequence spaces $\ell_p$ for $1\le  p<\infty$ different from 2, or $c_0$,

\item[(b)]  subspaces %of spaces  that are isomorphic to a subspace 
 of an $\ell_p$-sum  of finite-dimensional normed spaces, for $1<p<\infty$ $p\ne 2$, and,  in particular,  subspaces of  quotient spaces of $\ell_p$, for $1<p<\infty$, $p\ne 2$,

\item[(c)]   subspaces of asymptotic-$\ell_p$ spaces, $1 \le p \le \infty$, $p \ne 2$,

\item[(d)] subspaces of Asymptotic-$\ell_p$ spaces, $1 \le p \le \infty$, $p \ne 2$, in the sense of \cite{MMT},
%{\tt do we need to remind the definition here?}

\item[(e)]  subspaces of any Orlicz sequence space $\ell_M$ (where $M$ is an Orlicz function) such that $\ell_M$ does not contain a subspace isomorphic to $\ell_2$.
\item[(f)] subspaces   of $L_p, 2<p<\infty$,   that   do  not contain a subspace isomorphic to
$\ell_2$.
\end{itemize}
\end{theorem}

Their  method relies on an application of the classical Dvoretzky theorem, see e.g. \cite{GM}, which  
%{\tt (do you prefer to give the original reference \cite{D61}?)} 
 says that in every infinite dimensional Banach space for every 
natural number $m$ and every $\ep>0$, there exists a sequence 
$\{x_i\}_{i=1}^m$ in $X$ that is $(1+\ep)$-equivalent to the standard normalized basis of
$\ell_2^m$. It is   a very simple but key observation, that when $X$ is AT, then the first element $x_1$ in the above sequence can be chosen arbitrarily close to any element of the sphere of $X$.
Moreover, using compactness, given $x_1\in S_X$, there exists $x_2\in S_X$
 that is almost disjoint with $x_1$ (with  respect to a given  Schauder basis), and $\{x_1,x_2\}$ is $(1+\eps)$-equivalent  to the standard basis of two dimensional $\ell_2^2$.

It is not known whether this can be generalized to an arbitrary dimension $n$, that is, whether  for every $n\in \bbN$ every 
AT  space 
$X$ with a Schauder basis contains $n$  vectors that are mutually almost disjoint and 
$(1+\eps)$-equivalent to the standard normalized basis of
$\ell_2^n$.
%(note that this is true under additional assumptions on $X$,  c.f. Theorem~\ref{upper2} below)
However using  induction   the authors of \cite{DR} prove existence of block bases in  AT  spaces  that behave like the normalized basis of
$\ell_2^n$ for an arbitrary but   (sic!) fixed sequence of scalars.

 Recall that  a sequence $(x_i)_i$ of  vectors is a \textit{normalized block basis} if  $\|x_i\|=1$ ($i\ge1$), each vector $x_i$ is finitely supported,  and $x_1 < x_2 < x_3<\dots$, that is, for all 
 $i\ge 2$, $\max \supp x_{i-1} < \min \supp x_{i}$.

\begin{theorem} \label{prop: blockbasis}  
Suppose that $X$ has a Schauder basis and contains an infi\-ni\-te-dimensional subspace $Y$ which is almost transitive.  
Then, for any $\varepsilon>0$ and any sequence $(a_i)_i$  of nonzero scalars, there exists a normalized block basis $(x_i)_i$ in $X$
such that, for all $m \ge 1$, we have  
\begin{equation} \label{eq: blockbasis} 
\begin{split}
(1-\varepsilon)\|(a_i)_{i=1}^m\|_{\ell_2}\le \Big\|\sum_{k=1}^m a_k x_k\Big\| \le (1+\varepsilon)\|(a_i)_{i=1}^m\|_{\ell_2}.
\end{split} \end{equation} 
\end{theorem}

We stress that in   Theorem~\ref{prop: blockbasis} the block basis that satisfies \eqref{eq: blockbasis}  depends not only on $\ep>0$, but also on the selected scalar sequence $(a_i)_i$.
It turns out that this is powerful enough to imply several results on nonexistence of AT renormings. As an illustration we show how   it can be used to prove that no subspace of $\ell_p$, $1\le p<\infty$, $p\ne 2$,   admits an equivalent AT renorming.

 The argument is as follows: suppose that 
  a subspace $Y$  of $X$  admits an equivalent AT norm $\triple\cdot$. It is well-known that any equivalent norm 
on a subspace may be extended to an equivalent norm on the whole space, see e.g. \cite[p.\ 55]{FHHSPZ}. 
Then,  by Theorem~\ref{prop: blockbasis} applied to $(X,\triple\cdot)$   with the constant sequence 
$(a_i=1)_{i=1}^\infty$, there exists a disjointly supported sequence  
$(x_k)_k$ in $X$
such that  for all $n\in\bbN$, the norm $\triple{\sum_{k=1}^{n}  x_k}$
is $(1+\varepsilon)$-equivalent to $n^{1/2}$.
However,  $\triple\cdot$ is $C$-equivalent to $\|\cdot\|_{\ell_p}$, and, as is well known, every block basis of  $\ell_p$ is isometrically equivalent to the standard basis of $\ell_p$, see e.g. \cite{LT}, i.e. 
$ \triple{\sum_{k=1}^{n}  x_k}$
is $C$-equivalent to $n^{1/p}$, which gives the contradiction 
 when $p\ne 2$ and $n$ is large enough.

Essentially the same argument works for spaces  $X$ with a Schauder basis 
$(e_i)$ that satisfy $(p,q)$-estimates, where $1 < q \le p < \infty$, that is, such that there exists $C>0$ with
\begin{equation*} \frac{1}{C}\big(\sum_{k=1}^n \|x_k\|^p\big)^{1/p} \le \big\|\sum_{k=1}^n x_k\big\|  \le C\big(\sum_{k=1}^n \|x_k\|^q\big)^{1/q},
 \end{equation*} 
whenever $x_1<x_2<\dots<x_n$. Thus we have:

\begin{cor} \label{cor: pqestimates}
Suppose that a Banach space  $X$ with a Schauder basis $(e_i)$ contains a subspace $Y$ which admits an equivalent almost transitive norm. 
 If   $(e_i)$ satisfies $(p,q)$-estimates, then $q \le 2 \le p$. 
\end{cor}

Using similar reasoning and known properties of Banach spaces the authors of \cite{DR} obtain a list of classes of Banach spaces such that  none of their subspaces admits an equivalent AT renorming, stated in Theorem~\ref{list}.

\begin{rem} 
For a Banach space $X$, let
$$ \FR(X):= \{1 \le r \le \infty \colon \ell_r \text{ is finitely representable in $X$}\}.$$

The proof of 
 Theorem~\ref{prop: blockbasis}, is valid not only for the exponent 2, as stated, but (after the obvious modifications)  for any exponent $r \in \FR(Y)$,   where  $Y$ is an infinite dimensional AT subspace of $X$ (and $X$ has a Schauder basis). 
 %The proof yields  the same conclusion except \eqref{eq: blockbasis} should be replaced by
 %\begin{equation*} \label{eq: blockbasis2} 
 %\begin{split}
%(1-\varepsilon)\left(\sum_{k=1}^{m} |a_k|^r + |b|^r\right)^{1/r} &\le \left\|\sum_{k=1}^m a_k x_k + bx_{m+1}\right\| \\ 
%&\le (1+\varepsilon)\left(\sum_{k=1}^{m} |a_k|^r + |b|^r\right)^{1/r}.
%\end{split} 
%\end{equation*} 
% with the obvious modification for $r = \infty$.
\end{rem}

By the Maurey-Pisier theorem  \cite{MP}, this holds for all $r\in [p_Y,2] \cup \{q_Y\}$,
where $p_Y := \sup\{ 1 \le p \le 2 \colon \text{$Y$ has type $p$}\}$ and $ q_Y := \inf\{ 2 \le q < \infty \colon \text{$Y$ has cotype $q$}\}.$
For spaces with an unconditional basis, using results of Sari \cite{S04}, the authors of \cite{DR} then obtain a stronger version of  
Theorem~\ref{prop: blockbasis} for certain values of $r$.

\begin{theorem} \lb{upper2}
Suppose that $X$ has an unconditional basis $(e_i)_i$.  If $X$ has an equivalent
almost transitive renorming $\triple\cdot$, then   for $r=p_X$ and $r=q_X$ and for all $n \ge 1$ and $\varepsilon>0$, there exist disjointly  supported vectors $(x_i)_{i=1}^n\subset X$
 such that $(x_i)_{i=1}^n$ in the norm  $\triple\cdot$ is  $(1+\varepsilon)$-equivalent to the unit vector basis of $\ell_r^n$. 
\end{theorem}

\begin{small}
\noindent{\em Comments:} 
\begin{itemized}
\item[(1)]
The ``extreme'' cases in Theorem \ref{list}(a) namely subspaces of $\ell_1$ and $c_0$ are much easier and were proven earlier by Cabello S\'anchez   \cite[Theorem 2.1]{CS3}: {\em every space which is either Asplund or has the Radon-Nikod\'ym property with an AT re\-norming must be super-reflexive}. Actually these spaces and also those in entries {\rm (a), (b), (f)} lack convex transitive norms, by \cite[Corollary 6.9.]{becerra}. 
 \item[(2)]
 It is instructive to observe that  Theorem~\ref{prop: blockbasis}, when applied to the Haar basis of $L_p$,
does not contradict the fact that $L_p$ is AT. This is  because the unit vector basis of $\ell_2^m$ is
($1+\varepsilon$)-equivalent to a block basis of the Haar basis.  
\item[(3)]
A natural question, in the light of the third and fourth item of Theorem~\ref{list}, is whether every super-reflexive space which does not admit a subspace with an AT norm must contain an asymptotic-$\ell_p$ subspace.  The authors of \cite{DR}   answer this question negatively using as an example a space constructed in  \cite{CKKM}.

\item[(4)] This comment is in response to questions by Gilles Godefroy.%\m{series of comments on Property $(M)$} 
A Banach space $X$  has property $(M)$ (respectively, property 
$(m_p)$ for some $1\le p\le\infty$) if for each weak-null sequence $(x_n)$ and all $x,y\in X$ with $\|x\|=\|y\|$, 
$\limsup_n\|x+x_n\|=\limsup_n\|y+x_n\|$ (respectively, 
$\limsup_n\|x+x_n\|=\big(\|x\|^p+(\limsup_n\|x_n\|)^p\big)^{1/p}$).  These asymptotic  properties of the norm were introduced by Kalton in his study of $M$-ideals \cite{K-Mideals} and by Kalton and Werner in \cite{K+W}, respectively. The papers \cite{K-Mideals} and \cite{K+W} contain a wealth of examples and counterexamples as well as the ultimate connections between properties $(M)$, $(m_p)$, and $(M^*)$ and $M$-ideals of compact operators.

The presence of an equivalent norm with property $(M)$ has a considerable impact on the isomorphic structure of the underlying Banach space and has been used in \cite{K-Mideals} and \cite{K+W} in particular to obtain important characterizations of subspaces and quotients of $L_p$, the Schatten class $c_p$, and of $\ell_p$-sums of finite dimensional spaces.

 Dutta and Godard   \cite[Proposition 2.3]{D+G} showed that norms with property $(M)$ have ``optimal''  moduli of asymptotic uniform smoothness and asymtotic uniform convexity among all other equivalent renormings of the space.
 In \cite[Theorem~2.6]{D+G} they also identified an equivalent condition in terms of the ``optimal''  growth of the Szlenk indices of $X$ and $X^*$ for the separable reflexive space $X$ with property $(M)$ to be isomorphic to a subspace of an 
 $\ell_p$-sum of   finite dimensional normed spaces, where $p=\inf\{r: \sup_{\ep>0} \ep^r Sz(X^*,\ep)<\infty\}$. By Theorem~\ref{list}(b), in this situation if $p\ne 2$ no subspace of $X$ can be renormed to be AT.

Theorem~3.3 of \cite{K+W} says, in particular, that if a Banach space $X$ is separable, does not contain an isomorphic copy of $\ell_1$, and has property 
$(m_p)$ for some $1<p<\infty$, then $X$ is isomorphic to a subspace of an $\ell_p$-sum of  finite dimensional normed spaces. Thus, by Theorem~\ref{list}(b), if a Banach space $X$ is separable, does not contain an isomorphic copy of $\ell_1$, and can be equivalently renormed to satisfy property 
$(m_p)$ for some $1<p<\infty$, $p\ne 2$, then $X$ does not contain any infinite dimensional subspace that can be renormed to be AT.
However we do not know if the same holds true only under the assumption that $X$ is non-Hilbertian and has property $(M)$, or even whether a non-Hilbertian norm with property $(M)$ can be AT. We note that by \cite[Remark after Proposition~4.1]{K-Mideals} every reflexive Orlicz space which is not isomorphic to any $\ell_p$ can be renormed to satisfy property $(M)$ and simultaneously fail property $(m_p)$ for every $p$. By Theorem~\ref{list}(e) we know that if  of an Orlicz space $\ell_F$ does not contain a subspace isomorphic to $\ell_2$, then no subspace of $\ell_F$ can be renormed to be AT.

At this time the most that we can say about general spaces with property $(M)$ and AT subspaces  is the following.
It follows from \cite[Lemma~2.1]{DR} 
%\m{this lemma was originally stated and proved in this survey, and was later removed for the sake of brevity. Should we now include its statement without proof, just saying that it is the main tool for proving the main result which is Theorem 2.12?}
and \cite[Lemma~3.6 and the remark before it]{K-Mideals} that if $X$ has a Schauder basis and   property  
$(M)$, and if an infinite dimensional subspace of $Y$ of $X$ is AT, then  $Y$ contains a subspace   isomorphic to $\ell_2$ and, by 
\cite[Proposition~3.8]{K-Mideals}, every infinite dimensional subspace $Z$ of $Y$, for every $\ep>0$, contains a  subspace $E\subseteq Z$ with $d_{BM}(E,\ell_2)<1+\ep$.
\end{itemized}
\end{small}

\subsection{Spaces with multiple maximal bounded subgroups of ${\GL(X)}$} \lb{s:manymax}

We have seen in Section 2.1 that there exist Banach spaces  $X$ without maximal  bounded subgroups of $\GL(X)$. On the other hand, in this section, following \cite{DR}, we will show examples of spaces    with multiple different (i.e. non-conjugate)  maximal  equivalent renormings.

We say that two   equivalent norms 
$\|.\|$ and $\triple\cdot$ on   $X$ are  {\it conjugate} 
if there exists a bounded linear automorphism $T$ of $X$ such that
$\|x\|=\triple{Tx}$ for all $x \in X$.  Note that in this case, $T$ induces an bilipschitz homemorphism between the unit spheres of $X$ under the two norms. 
On the other hand we say that the groups
 $\Isom(X,\|.\|)$ and $\Isom (X,\triple\cdot)$ are  conjugate if they are conjugate as   subgroups of $\GL(X)$, that is, 
if there exists a bounded linear automorphism $L$ of $X$ such that
$$\Isom(X,\|.\|)=L^{-1}\Isom (X,\triple\cdot)L.$$
Note that if two norms are conjugate then their respective isometry groups are conjugate, but the converse does not hold (the isometry group for the two norms could e.g. be trivial and therefore equal without the norms being conjugate). However the two notions are equivalent in the following important case, which will simplify considerably some of our proofs.

\begin{lemma}
Two convex transitive norms on a Banach space are conjugate if (and only if) they have conjugate isometry groups.
\end{lemma}

\begin{proof}
Let $\|\cdot\|$ and $|\cdot|$ be CT norms on $X$ and assume that 
$\Isom(\|\cdot\|)$ and $\Isom(|\cdot|)$ are conjugate through $L\in\operatorname{GL}(X)$ so that $T$ is an isometry of $\|\cdot\|$ if and only if $\tilde T=LTL^{-1}$ is  an isometry of $|\cdot|$. Pick $x_0\in X$ so that 
$\|x_0\|=1$. Multiplying $L$ by a positive constant if necessary, we can and do assume WLOG that $|Lx_0|=1$. Then for all $T\in \Isom(\|\cdot\|)$, 
$\tilde{T}=LTL^{-1}\in \Isom(|\cdot|)$ and thus $|L(Tx_0)|= |LT(L^{-1}Lx_0)|=|\tilde{T}Lx_0|=|Lx_0|=1$. Hence, by CT, for all
 $x\in X$, $|Lx|\le \|x\|$.

Let us then check that the CT of $\|\cdot\|$ entails that $L:(X,\|\cdot\|)\to (X,|\cdot|)$ is contractive. By symmetry we shall also have that $L^{-1}:(X,|\cdot|)\to (X,\|\cdot\|)$ is contractive  and so  $\|\cdot\|$ and $|\cdot|$ are conjugate. %Pick $x\in X$ so that $\|x\|=1$ and assume WLOG that $|Lx|=1$.
As the unit ball of $\|\cdot\|$ is the closed convex hull of the orbit of $x$ under the action of $\Isom(\|\cdot\|)$ it suffices to see that if $y=Tx$ for some $T\in \Isom(\|\cdot\|)$, then $|Ly|=1$. Which is easy: $Ly=LTx= LTL^{-1}Lx=\tilde{T}Lx$ and $\tilde{T}$ is an isometry of $|\cdot|$. 
%Pick $x\in X$ so that $\|x\|=1$ and assume WLOG that $|Lx|=1$.
%Let us then check that the CT of $\|\cdot\|$ entails that $L:(X,\|\cdot\|)\to (X,|\cdot|)$ is contractive. By symmetry we shall also have that $L^{-1}:(X,|\cdot|)\to (X,\|\cdot\|)$ is contractive  and so  $\|\cdot\|$ and $|\cdot|$ are conjugate. %Pick $x\in X$ so that $\|x\|=1$ and assume WLOG that $|Lx|=1$.
%As the unit ball of $\|\cdot\|$ is the closed convex hull of the orbit of $x$ under the action of $\Isom(\|\cdot\|)$ it suffices to see that if $y=Tx$ for some $T\in \Isom(\|\cdot\|)$, then $|Ly|=1$. Which is easy: $Ly=LTx= LTL^{-1}Lx=\tilde{T}Lx$ and $\tilde{T}$ is an isometry of $|\cdot|$. \m{I corrected the proof. Assume first that $|Lx|=1$ for some $\|x\|=1$ to claim contractive (otherwise you could always replace $L$ by $\alpha L$ for $\alpha>0$)}
\end{proof}
 
Constructions in \cite{DR} use  vector valued spaces defined as follows.

Let $X$  be a Banach space with a 1-unconditional basis  
$E=(e_k)_{k\in \bbN}$ and $(Y_k)_{k\in \bbN}$ be Banach spaces. Then $$Z=\Big(\sum_{k\in\bbN}\oplus Y_k\Big)_E$$
is the space of all sequences $(z_k)_{k\in \bbN}$ such that for all $k$, $z_k\in Y_k$, and 
$$\|(z_k)_k\|_Z:=\Big\| \sum_{k\in\bbN}\|z_k\|_{Y_k}\ e_k\Big\|_X$$
is finite. When $E$ is   a standard basis of $\ell_p$, we will sometimes write 
 $Z=(\sum_{k\in\bbN}\oplus Y_k)_{\ell_p}$ to mean the same as 
 $Z=(\sum_{k\in\bbN}\oplus Y_k)_E$.

Rosenthal \cite{R86} characterized
isometries of  spaces of this form in the case when all spaces $Y_k$  are hilbertian and  $X$  is a {\it pure}  space with a normalized 1-unconditional basis. 
%Rosenthal   \cite{R86} defined a concept of a {\it pure Banach space} as follows:   
A Banach space $X$ with a normalized 1-unconditional basis $\{e_\g\}_{\g\in \G}$ is  called {\it impure} if there exist $j\ne k$ in $\G$ 
%\m{F: 
%It is clear that the notions of pure/impure refer to normalized 1-unconditional bases, not spaces, and actually the results would be more clearly and elegantly stated taking this into account: for instance Th~\ref{FHS} below could begin as {\em Let $E=(e_\gamma)_{\gamma\in\G}$ be a pure basis...} and the same for Th~\ref{nmax} 
%V. Felix, I checked Rosenthal's paper and the original definition of purity is for the space, without ref to a basis (having or not a rank 2 skew-Hermitian operator). Then, by Rosenthal Corollary 3.4  you can characterize purity from the basis: Corollary 3.4 is an equivalent definition. So it is a bit messy and I am not sure it is worth defining a pure basis when it is actually not defined in the original paper of Rosenthal.
%For me we could leave it as it is. But do as you wish. Btw, I checked and the def. written as it is is correct (even if Godefroy had doubts)}
so that $(e_j,e_k)$ is isometrically equivalent to the usual basis of 2-dimensional $\ell_2^2$, and for all $x, x'\in \Span(e_j,e_k)$ with $\|x\|=\|x'\|$ and for all $y\in\Span\{e_m:m\ne j,k\}$ we have  $\|x+y\|=\|x'+y\|$ (\cite{R86} Corollary 3.4).%\m{The definition is correct}
Otherwise the space is called {\it pure}. 
The  space $\ell_p$, $1\le p<\infty$, $p\ne2$, is a natural example of a pure space. 
 %is pure, as well as any Banach space with a 1-symmetric basis such that 
%$\|e_1+e_2\|\ne\sqrt{2}$.

 Rosenthal proved the following result.

\begin{theorem} \cite[Theorem~3.12]{R86}\label{FHS}
Let $X$  be a pure space with a 1-un\-con\-di\-tional basis  $E=\{e_\g\}_{\g\in \G}$, $(H_\g)_{\g\in \G}$ be Hilbert spaces all of dimension at least 2, and let $Z=(\sum_\G\oplus H_\g)_E$. 

Let $P(Z)$ denote the set of all bijections $\s:\G\to\G$ so that
\begin{enumerate}
\item[(a)] $\{e_{\s(\g)}\}_{\g\in \G}$ is isometrically equivalent to $\{e_\g\}_{\g\in \G}$, and
\item[(b)] $H_{\s(\g)}$ is isometric to $H_\g$ for all $\g\in \G$.
\end{enumerate}

Then $T:Z\to Z$ is a surjective isometry if and only if there exist $\s\in P(Z)$ and surjective linear isometries $T_\g: H_\g\to H_{\s(\g)}$, for all $\g\in \G$,  so that for all $z=(z_\g)_{\g\in \G}$ in $Z$, and for all $\g\in \G$, 
\begin{equation}\lb{isofiber}
(Tz)_{\s(\g)}=T_\g(z_\g).
\end{equation}  
\end{theorem}

Theorem~\ref{FHS} is valid for both real and complex spaces. For separable complex Banach spaces it was proved earlier  by Fleming and Jamison \cite{FJ74}, cf. also \cite{KW76}.

 Dilworth and  Randrianantoanina \cite{DR}, using 
Theorem~\ref{FHS},  described a countable number of  different equivalent maximal norms on  every Banach space with  a 1-symmetric basis, which is not isomorphic to $\ell_2$.

\begin{theorem} \lb{nmax}
Suppose $X=\ell_p$, $1\le p<\infty$, $p\ne2$, or, more generally, $X$ is a pure Banach space with a   1-symmetric  basis $E=\{e_k\}_{k=1}^\infty$,  and $X$ is not isomorphic to $\ell_2$.
Then $X$ admits countably many mutually non-conjugate equivalent maximal renormnings.

Namely, for
 $n\in\bbN, n\ge 2$, let 
$$Z_n=Z_n(X)=(\sum_{k=1}^\infty \oplus H_k)_E,$$
 where, for all $k\in\bbN$, $H_k$ is isometric to $\ell_2^n$.
 Then  $Z_n$ is  isomorphic to $X$, the isometry group of $Z_n$ is maximal, and,  
  if $n \ne m$, the groups $\Isom(Z_n)$  and $\Isom(Z_m)$ are not conjugate in   $\GL(X)$.
\end{theorem}

\begin{proof}[Idea of proof] 
It is easy to see that $Z_n$ is isomorphic to the direct sum of $n$ copies of $X$ and hence isomorphic to $X$ itself since $X$  has a symmetric basis.

By Theorem~\ref{FHS}, all isometries of $Z_n$ have form \eqref{isofiber}, and, since the basis is 1-symmetric and all $H_k$ are isometric to each other, the set $P(Z_n)$ is equal to the set of all bijections of $\bbN$.

We claim that the group $\Isom(Z_n)$ is maximal. Let's first consider the case when $\s$ is the identity  of $\bbN$ and, for all $k$, $T_k\in\Isom(H_k)$. Since 
$\Isom(H_k)=\Isom(\ell_2^n)$ is the largest possible group of isometries of any $n$-dimensional space, it is impossible to renorm each $H_k$ to increase the isometry group of $Z_n$.  So how can we renorm $Z_n$ to introduce additional isometries? 

A first natural idea that comes to mind is to ``glue'' two or more, but finitely many, fibers of $Z_n$ and equip this new larger fiber with the  norm that has the largest possible isometry group, i.e. the $\ell_2$ norm. Say, if we put for all $k\in \bbN$, $\widetilde{H}_k=H_{2k-1}\oplus_2H_{2k}$ and consider
 $\widetilde{Z}_n=(\sum_{k=1}^\infty \oplus \widetilde{H}_k)_{E}$. Then, for all $k$, $\dim \widetilde{H}_k=2n$, and if $\s={\bf I}_\mathbb N$, there exists an isometry $\widetilde{T}_k$ of $\widetilde{H}_k$ so that 
 $\widetilde{T}_k(H_{2k-1})$ intersects both $H_{2k-1}$ and $H_{2k}$, so the operator $\widetilde{T}:\widetilde{Z}_n\to\widetilde{Z}_n$ defined by 
 $\widetilde{T}((\widetilde{z}_k)_k)=(\widetilde{T}(\widetilde{z}_k))_k$ is an isometry of $\widetilde{Z}_n$ but not of ${Z}_n$.
 
 On the other hand, since the basis $E$ is 1-symmetric,  if we consider a permutation $\s$ of $\bbN$ so that, say,
 $\s(1)=3$ and $\s(2)=5$, and arbitrary isometries $T_k:H_k\to H_{\s(k)}$, then the operator $T:Z_n\to Z_n$ such that for all $z=(z_k)_k\in Z_n$,
 $(Tz)_{\s(k)}=T_k(z_k)$ is an isometry of $Z_n$.  However $T$ is not an isometry of 
 $\widetilde{Z}_n$, since, by Theorem~\ref{FHS}, any isometry of 
 $\widetilde{Z}_n$ maps the fiber $\widetilde{H}_1$ either to itself or onto another fiber and we have that $T(\widetilde{H}_1)$ intersects both 
 $\widetilde{H}_2$ and $\widetilde{H}_3$. 
 
 Hence we have that $\Isom(\widetilde{Z}_n)\not\subseteq\Isom(Z_n)$ and 
 $\Isom(\widetilde{Z}_n)\not\supseteq\Isom(Z_n)$, and thus our ``gluing'' of fibers failed to produce a space with a larger isometry group. 
 
 It follows from 
 \cite{R86} that if $\widetilde{Z}_n=(Z_n, |\!|\!|\cdot|\!|\!|)$ is an equivalent renorming of $Z_n$  so that $\Isom(\widetilde{Z}_n)\supseteq\Isom(Z_n)$, then  $\widetilde{Z}_n=(\sum_{k=1}^\infty \oplus \widetilde{H}_k)_{E}$, where each  new fibers $\widetilde{H}_k$ is an  
  $\ell_2$ sum of a certain finite subcollection of the original fibers. The idea of the remaining part of the proof is same as above.
  
The fact that isometry groups    $\Isom({Z}_n)$ are mutually non-conjugate follows   from \eqref{isofiber}, since  for different values of $n$ the dimensions of hilbertian fibers are different and  $E$ is pure, see \cite[Proposition 3.4]{DR} and \cite[Theorem~2]{R86}. 
\end{proof}

The construction of 
Theorem~\ref{nmax} can be generalized to describe a continuum of different  (pairwise non-conjugate) maximal    renormings of  Banach spaces $Z$ that have  the form
$
Z = (\sum_{k=1}^\infty \oplus \ell_2^{n(k)})_E,
$
where $E=(e_k)_{k=1}^\infty$ is a  1-symmetric  basis of a pure Banach space $X$ that is not isomorphic to $\ell_2$.
 It follows from 
  the Pe\l czy\'nski decomposition method that the space $Z$   is isomorphic to $X$     if, for example, $X=\ell_p$, with $1<p<\infty$, or if $X=U$,  Pe\l czy\'nski's space with a universal unconditional basis   \cite{P69} mentioned  in Section \ref{sec:classical-isom}, see \cite[Theorem 3.7]{DR} for details. %\m{A small typo corrected ($k\to n(k)$); and a more precise reference}

 Note that, as $Z$ is a separable Banach space,  the collection of all equivalent norms on $Z$ has cardinality $\mathfrak{c}$. Hence the  maximal cardinality of a    collection of pairwise non-conjugate maximal bounded subgroups of $\GL(Z)$ 
  is exactly equal to $\mathfrak{c}$.

\begin{cor}\label{cor:nmax} Each of the spaces $\ell_p$, for $1<p <\infty$, $p \ne 2$, and the space $U$ with a universal unconditional basis, admits a continuum of equivalent  renormings whose isometry groups are maximal and   pairwise non conjugate
in  the group of bounded isomorphisms. \end{cor}

The above results suggest the following   questions:

\begin{prob}\label{3problems}~\
\begin{itemize}
\item[(a)]  Let $\mathcal H$ be a Hilbert space. Is the unitary group the  unique, up to conjugacy,  maximal bounded subgroup of $\GL(\mathcal H)$? 
\item[(b)]  Does there exist a separable Banach space $X$  with a unique, up to conjugacy,   maximal  bounded subgroup of $\GL(X)$? 
\item[(c)]  If yes, does $X$ have to be 
 isomorphic to a Hilbert space?
 \end{itemize}
\end{prob}
\

\begin{small}
\noindent{\em Comments:} 
\begin{itemized}
 \item[(1)] %Since the unitary group is maximal,%
 Problem \ref{3problems}(a) may be reformulated as asking whether the Hilbert space admits a maximal, ``non-unitarizable" bounded group of automorphisms. See Section \ref{3} for more about this question. 
 \item[(2)] 
Theorem~\ref{nmax} applies in particular to 
the space $S(T^{(2)})$, the  symmetrization of the 2-con\-ve\-xified Tsirelson space, see \cite{CS89}. Indeed,  it is known that $S(T^{(2)})$ does not contain $\ell_2$, and
it is easy to verify that for all $k,l \in \bbN$, $\|e_k+e_l\|_{S(T^{(2)})}=1$, and thus the standard basis of $S(T^{(2)})$  is pure.  It is clear that the   renormings of $S(T^{(2)})$ described in 
Theorem~\ref{nmax} are not AT. 
 
It is known that any symmetric weak Hilbert space is isomorphic to a Hilbert space, but in some sense the space $S(T^{(2)})$ is very close to a weak Hilbert space, see \cite[Note A.e.3 and Proposition A.b.10]{CS89}.
We do not know the answers to the following problems:
\end{itemized}
 \end{small}

\begin{prob}
Does the space $S(T^{(2)})$ admit an AT renorming?  Does there exist a  symmetric space not isomorphic to $\ell_2$ which admits an AT renorming?
 \end{prob}

  \subsection{Spaces with multiple almost transitive norms}\label{sec:smAT}
 
In this section we consider the   %question of
 existence of different maximal renormings of the space $L_p$, for $p\in[1,\infty)$ different from $2$. We show that an analogue of Theorem~\ref{nmax} holds for $L_p$, and in this case it gives  a countable family of mutually non-conjugate equivalent almost transitive norms. All the results in this section seem to be new and we have included (more or less) full proofs. %
%Our task is {\em now} greatly simplied by the following remark:

%\begin{lemma}
%Two convex transitive norms on a Banach space are conjugate if (and only if) they have conjugate isometry groups.
%\end{lemma}

%\begin{proof}
%Let $\|\cdot\|$ and $|\cdot|$ be CT norms on $X$ and assume that $\Isom(\|\cdot\|)$ and $\Isom(|\cdot|)$ are conjugate through $L\in\operatorname{GL}(X)$ so that $T$ is an isometry of $\|\cdot\|$ if and only if $\tilde T=LTL^{-1}$ is  an isometry of $|\cdot|$.
%Let's check that the CT of $\|\cdot\|$ entails that $L:(X,\|\cdot\|)\to (X,|\cdot|)$ is contractive. By symmetry we also have that $L^{-1}:(X,|\cdot|)\to (X,\|\cdot\|)$ is contractive  and so  $\|\cdot\|$ and $|\cdot|$ are conjugate. Pick $x\in X$ so that $\|x\|=1$ and assume WLOG that $|Lx|=1$. As the unit ball of $\|\cdot\|$ is the closed convex hull of the orbit of $x$ under the action of $\Isom(\|\cdot\|)$ it suffices to see that if $y=Tx$ for some $T\in \Isom(\|\cdot\|)$, then $|Ly|=1$. Which is easy: $Ly=LTx= LTL^{-1}Lx=\tilde{T}Lx$ and $\tilde{T}$ is an isometry of $|\cdot|$. 
%\end{proof}

\begin{theorem}\label{maxcapitalLp} 
For $p\in[1,\infty)$ different from $2$
the space $ L_p$ has at least countably many non-conjugate almost transitive norms.
\end{theorem}%\m{B: should we mention:This is a new result, so we include a full proof.}

\begin{proof}
For each $n\ge 1$ the space $L_p$ is isomorphic to $L_p(H_n)$,  where $H_n$ is the $n$-dimensional Hilbert space. In \cite[Theorem 2.1]{GJK} it was proved  that the standard norm on $L_p(H_n)$ is AT. 
The (AT) 
% We show that the following  norms induce isometry groups which are mutually not conjugate in $\GL(L_p)$: the
% usual $L_p$ norm, and
 norms in $L_p$ induced by an isomorphism onto $L_p(H_n)$ are, however, not conjugate in $\operatorname{GL}(L_p)$ for different values of $n$ because $L_p(H_n)$ is isometric to  $L_p(H_m)$ only if $n=m$, by results of Cambern and Greim; see \cite[8.2.11. Theorem]{FJ2}.
\end{proof}
%\m{B: I don't want to keep a separate lemma, I copied it here because I think that some portion of its proof should be included in the proof of Theorem 2.21

%F: in the end we have used it 3 times so that I think it is better placed before the ``main results'' so that the section flows very naturally...}

The same occurs in $C[0,1]$. Indeed Aizpuru and Rambla proved in \cite[Proposition 6.2]{AR} that $C_0(P_*,H_n)$ is AT for all $n\geq 2$ no matter  which  field of scalars one considers. While the isometric type of these spaces effectively depends on $n$, by a classical result of Jerison \cite[7.2.16. Theorem]{FJ2}, they are all 
%But the spaces $C_0(P_*,H_n)$ are all 
 isomorphic to $C[0,1]$ by Miljutin's Theorem; see \cite[Section 4.4]{AK} for a polished proof. Other ``individual'' AT renormings of $C[0,1]$ arise from \cite[Theorem 3.4]{CS7}, \cite[Examples 2.4 and 3.2]{c-edin} and \cite[Corollary 6.9]{FLMT}.

The Garbuli\'nska space provides a more spectacular example:

\begin{theorem}\label{th:K has c}
The Garbuli\'nska space $\mathcal K$ has a continuum of mutually non-conju\-gate almost transitive norms.
\end{theorem} 
%\m{F: Beata, I kept this comment in case you want to use it just before Corollary 2.19 (V: NB: numbers might have changed, so Corollary 2.19 might be some other number in this version).   \\

%The result is optimal: a separable Banach space has at most a continuum of (not necessarily equivalent) non-conjugate norms because there exist exactly a continuum of mutually non-isometric separable Banach spaces (think of the subspaces of $C[0,1]$).}

\begin{proof} As remarked in \cite[p.~1551]{2cm}, $\mathcal K$ is the peskiest Banach space there is. In particular $\mathcal K$ is isomorphic to each of the spaces $L_p(\mathcal{K})$ for $1\leq p<\infty$. To see this we observe that $L_p(\mathcal{K})$ has the BAP for all $p$, $1\le p<\infty$, and therefore it is isomorphic to a complemented subspace of $\mathcal K$ since the latter is complementably universal for the BAP.
On the other hand, any space $X$ is 1-complemented (as the space of constant functions) in $L_p(X)$ for any $1\leq p\leq \infty$ by means of the ``obvious'' projection $P(f)=\int_0^1f(t)dt$, where the integral is taken in the Bochner sense.
%if $f=\sum_{i=1}^n1_{A_i}\otimes x_i$ is simple in $L_p(X)$ we can define $P(f)=\sum_{i=1}^n|{A_i}| x_i$;
%it is clear that $\|Pf\|_X\leq \|f\|_p$ for $1\leq p\leq \infty$. Since the simple functions are dense in $L_p(X)$ we see that $P$ extends to a contractive projection of $L_p(X)$ onto its subspace of constant functions. 
 An easy application of Pe\l czy\'nski decomposition method yields $\mathcal K\simeq L_p(\mathcal{K})$ for $1\leq p<\infty$.

Next we remark that $L_p(\mathcal{K})$ is AT for $1\leq p<\infty$ by the result of Greim, Jamison and Kami\'nska already mentioned. For $1\leq p<\infty$, let $|\cdot|_p$ denote the AT renorming of $\mathcal K$ induced by some (fixed) isomorphism $\mathcal K\to  L_p(\mathcal{K})$. We claim that $(\mathcal K, |\cdot|_p)$ and  $(\mathcal K, |\cdot|_q)$ cannot be isometric if $p\neq q$. To see this recall that an $L^p$-projection on a Banach space $X$ is a projection $P$ such that $\|x\|^p=\|Px\|^p-\|x-Px\|^p$ for all $x\in X$. It is clear that $ L_p(\mathcal{K})$ (and so $(\mathcal K, |\cdot|_p)$) has {\em non-trivial} $L^p$-projections (think of multiplication by characteristic functions). But the only Banach space that admits nontrivial $L^p$-projections for two different values of $p$ is $\ell_1^2\approx \ell_\infty^2$ (real case; see \cite[Main theorem]{beh}) from which the claim follows.
\end{proof}

By taking ultrapowers of the preceding examples, and using general representation results to describe the corresponding ultrapowers if necessary, we obtain: 
 
%There exist non-separable Banach spaces not isomorphic to a Hilbert space with at least countably many pairwise non-conjugate  equivalent transitive norms.

\begin{cor}\lb{manytransitive}
Let $\mathcal U$ be a free ultrafilter on the integers.
\begin{itemize}
\item[{\rm (a)}] For each $p\in[1,\infty)$ different from $2$ the ultrapower $(L_p)_\mathcal U$ has countably many pairwise non-con\-jugate transitive norms.
\item[{\rm (b)}] $(C[0,1])_\mathcal U$ has countably many pairwise non-conjugate transitive norms.
\item[{\rm (c)}] $\mathcal K_\mathcal U$ has a continuum of pairwise non-conjugate transitive norms.
\end{itemize}
\end{cor}

\begin{proof}%\m{corrected as (c)(b)(a) as required}[Sketch of the proof] (c)
The case of  $\mathcal K_\mathcal U$ is clear because ultrapowers of spa\-ces with nontrivial $L^p$-projections have again nontrivial $L^p$-projections so we can use the ultrapowers of the norms in  Theorem~\ref{th:K has c}.

(b) Note that $(C[0,1])_\mathcal U\simeq C_0(P_*)_\mathcal U$ by Miljutin's Theorem. It is known that if $L$ is a locally compact space, then $C_0(P_*)_\mathcal U$ is isometrically isomorphic (even as a ring) to $C_0(L^\mathcal U)$, with $L^\mathcal U$ a ``huge" locally compact space. Explicit descriptions are available. Now, the point is that for fixed $n$,  the ultrapower $C_0(L,H_n)_\mathcal U$ is isometric with $C_0(L^\mathcal U,H_n)$. 
This can be proved in many ways. Perhaps the simplest one is to identify $C_0(L,H_n)$ with the injective tensor product of $C_0(L)$ and $H_n$. That said, we have that $C[0,1]_\mathcal U$ is isomorphic to each of the transitive spaces $C_0(P_*,H_n)_\mathcal U=C_0(P_*^\mathcal U,H_n)$ which cannot be isometric for different values of $n$ because of Jerison's result: \cite[7.2.16 Theorem]{FJ2}: {\em If $Y$ is a strictly convex Banach space, then $(X,Y)$ has the Banach-Stone property for any Banach space $X$. If both $X$ and $Y$ are
strictly convex, then $(X,Y)$ has the strong Banach-Stone property} (which in particular implies that if $C_0(L_1,X)$ is isometric to $C_0(L_2,Y)$ then $L_1$ is homeomorphic to $L_2$ and $X$ is isometric to $Y$.

(a) The $L_p$ case is a bit trickier. Fix $p \in [1,\infty)$ and use that 
$(L_p)_\mathcal U$ is isometric, even as a lattice, to $L_p(\mu)$ for some ``huge" measure $\mu$; see Heinrich's \cite[Theorem 3.3(ii)]{hein}. In any case one can assume $\mu$ strictly localizable, by a result of Maharam ({\em cf}. 
Lacey \cite[Corollary on p. 137]{lacey}). After that show that for each fixed $n$ the space $(L_p(H_n))_\mathcal U$  is isometric to $L_p(\mu, H_n)$; one can use a basis of $H_n$ or a tensor product argument. Finally, dig into the details of Section~8.2 of Fleming-Jamison to check that \cite[8.2.11 Theorem]{FJ2} survives if the Hilbert spaces are finite-dimensional and one considers strictly localizable (instead of $\sigma$-finite) measures.
\end{proof}

\ 

\begin{small}
\noindent{\em Comments} 
\begin{itemized}
\item[(1)]   
Rather curiously, we do not know whether the space $\mathcal K$ ``itself'' (i.e. in the Garbuli\'nska norm) has non-trivial $L^p$-projections for some (necessarily unique) $p\in[1,\infty]$.
 \item[(2)]
A separable version of Corollary \ref{manytransitive} is clearly out of reach, as it would require an answer to the Mazur rotations problem in its isometric or isomorphic version. See Section~\ref{3} for discussion about transitive renormings of the Hilbert space.
\item[(3)] Regarding Theorem~\ref{maxcapitalLp}, we have been unable to decide whether the  isometry groups of the spaces $L_p(H_n)$ for different $n$'s are isomorphic either in the purely algebraic sense or when they are equipped with  SOT or the norm topology. In the real case one can prove that for any $n\ge 2$, $\Isom(L_p)$ is not topologically isomorphic to  
$\Isom(L_p(H_n))$ in the norm topologies because if $T,L$ are different isometries of any $L_p(\mu)$, then $\|T-L\|\geq 2^{1/p}$. Thus $\Isom(L_p)$ is discrete in the norm topology, while  for each $n\ge 2$, 
$\Isom(L_p(H_n))$ is not as it
contains $\Isom(H_n)$.
%A similar argument to Theorem \ref{maxcapitalLp} shows that the isometry group of any real $L_p(m)$ is discrete in the norm topology.
\item[(4)] It is perhaps worth noticing the following application: if $X$ is a {\em real} Lindenstrauss space (that is, $X^*$ is isometric to $L_1(\mu)$ for some measure $\mu$), then $\Isom(X)$ is discrete in the norm - just use the estimate in Comment (3) together with the natural isometric embedding of $\Isom(X)$ into $\Isom(X^*)$.  
%\m{Hint given}
In this case the isometries are as far as they can be: $\| T - L \| = 2$  unless $T = L$. This applies, in particular to the Gurariy space.

%\m{B: In any case, I would like to include a question whether K admits renormings whose isometry group is NOT contained in any maximal
%bounded subgroup of GL(K), or is this known?
%}
\end{itemized} 
\end{small}

\subsection{Isometry groups not contained in any maximal bounded subgroup of the isomorphism group}\lb{secnonmax}

In  \cite{DR}  Dilworth and   Randrianantoanina show\-ed that 
Problem~\ref{W2} can have a negative answer even if $\GL(X)$ contains many maximal bounded subgroups. Namely they proved (constructively):
 %demonstrated continuum different equivalent renormings of spaces
 % $\ell_p$, $1 < p<\infty$, $p\ne2$, so that their groups of isometries are not contained in any maximal bounded subgroup of $GL(\ell_p)$.

\begin{theorem}\label{Bnonmax} 
Each of the spaces $\ell_p$ for $p\in[1,\infty)$ different from $2$ and $U$ has a  continuum of  pairwise non conjugate  renormings none of whose isometry groups  is  contained in any maximal bounded subgroup of the isomorphism group  of $\ell_p$.
\end{theorem}  
  
Compare with Corollary~\ref{cor:nmax}. 
The idea of the proof is similar to the proof of Theorem~\ref{nmax}. The essential difference is that this time the $E$-sums  are taken of sequences of Hilbert spaces that are not of the same dimension, but are all of different dimensions and, in addition, sums of dimensions of any two finite subcollections of fibers are never   equal to  each other, see \cite[Section~4]{DR} for details.

\begin{prob}
Does there exist a separable Banach $X$ space so that every bounded subgroup of $\GL(X)$ is contained in some maximal  bounded subgroup of $\GL(X)$?
Is this true for $X=L_p$? 
\end{prob}

\begin{small}
\noindent{\em Comments} 
\begin{itemized}
\item[(1)] The conclusion of Theiorem~\ref{Bnonmax}  is also true, for example, for
  the 
2-convexi\-fied Tsirelson space $T^{(2)}$ %, the  Pe\l czy\'nski's space $U$ with a universal unconditional basis, 
 and spaces of the form $(\sum_{n=1}^\infty \oplus \ell_2^n)_{E}$, where $E$ is   symmetric,  pure, and 
not isomorphic to a Hilbert space.
We note that $T^{(2)}$  is a weak Hilbert space. 

\item[(2)] We do not know whether $T^{(2)}$ %\m{B: include this also in section 5 on open problems}
or general weak Hilbert spaces, other than the Hilbert, have a  maximal  bounded subgroup of $\GL(X)$. 
\end{itemized}
\end{small}

\subsection{Almost-transitivity,  subspaces, and stabilizers}
\lb{sub-stabilizers}

%We say that a subspace $Y$ of a Banach space $X$ is {\it 1-complemented} in $X$ if there exists a norm one projection $P$ from $X$ onto $Y$.

In the Hilbert space case we may note that  the unitary group acts transitively not only on the sphere of $X$, but also on spheres of all infinite dimensional subspaces. We may ask to which extent this characterizes the Hilbert space. Some results in this direction were obtained in \cite{DR} as a consequence of 
Theorem~\ref{list} and known properties of Banach spaces.

\begin{prop} \lb{pin2infty}
Let $X$ be a subspace of $L_p$, $2<p<\infty$,  so that every subspace  of $X$ admits an almost transitive renorming, then $X$ is isomorphic to $\ell_2$.
\end{prop}

In view of Proposition~\ref{pin2infty} (see also comments to this section) it is natural to ask:

\begin{prob}\lb{pr:allsubsp}
  Suppose that every subspace of a Banach space $X$ admits  an  almost transitive renorming. Is $X$ isomorphic to  a Hilbert space?
\end{prob}

Next we turn to   some sufficient conditions on hyperplanes
 (i.e. 1-codi\-men\-sional subspaces) which together with almost transitivity of $X$ imply that $X$ is isometric to a Hilbert space.
  The first result that we want to mention here is due to J. Talponen, who generalized an earlier result of
    Randrianantoanina  \cite{beata2} that all real AT spaces that have a 
     1-comple\-mented hyperplane are isometric to a Hilbert space.

\begin{theorem}{\cite[Theorem 2.3]{T07}}\lb{hyperplane}
 Suppose that $X$ is a real almost transitive Banach space
and that for
each $\ep>0$, $X$ contains a $(1+\ep)$-complemen\-ted hyperplane. Then $X$ is isometric to a Hilbert space.
\end{theorem}

Another type of condition that is natural to consider is that the group 
$\Isom(X)$ acts almost transitively on some hyperplane on $X$. This by itself is not sufficient to conclude that $X$ a Hilbert space, since Talponen \cite{T09} showed that the isometry  group of $L_1$ 
%$\Isom(L_1[0,1])$ 
acts almost transitively on the hyperplane 
$M=\{f\in L_1: \int_0^1 f=0\}$ (and leaves it invariant). 
Thus some additional conditions are necessary. 

The results in the remaining part of this section are new, so we include their full proofs.

If $x_0 \in S_X$ then we define
$${\rm Stab}_{x_0}(X)=\{T \in \Isom(X): Tx_0=x_0\}.$$
This is a  closed subgroup of the isometry group, which under some natural hypotheses, acts on the hyperplane $H_{x_0}$ generated by the norming functional of $x_0$. 
%If $Y$ is a subset of $X$, $\Isom_Y(X)$ is the subgroup of $\Isom(X)$ of isometries leaving $Y$ invariant ($T(Y)=Y$). If $Y$ consists of one point $y$, we denote this subgroup $\Isom_y(X)$ ($={\rm Stab}_{y}$).

We investigate the case where the stabilizers act   almost transitively on the appropriate hyperplane and obtain a partial answer to Mazur rotations problem.

Recall that by a theorem of Mazur \cite[Theorem 8.2]{FHHSPZ}, the norm is 
G\^ateaux differentiable on a dense $G_\delta$ subset of $S_X$ when $X$ is a separable Banach space. Thus, if $X$ is separable and transitive, then the norm is    G\^ateaux differentiable at every point of $S_X$ and also  strictly convex; see Section~\ref{sec:strict}.

Note that the G\^ateaux differentiability %\m{Thank you!}
of the norm   at some $x \in S_X$, supported by the (unique) normalized funcional $\phi$, implies that the group ${\rm Stab}_x(X)$ leaves invariant the hyperplane $H_0=\Ker\phi$: indeed from $Tx=x$ it is immediate to deduce $T^* \phi=\phi$ and therefore that $H_0$ is invariant.

Conversely, strict convexity implies the following:

\begin{lemm}\lb{L:pmx}
 Assume$X$ is a real Banach space and the norm in strictly convex at $x$. Let $\phi$ be a support functional for $x$, and $H_0=\Ker \phi$. If an isometry $T$ satisfies $T(H_0)=H_0$, then $Tx=\pm x$.
 \end{lemm}

\begin{proof} Let $T\in \Isom(X)$ with  $T(H_0)=H_0$.
Since
$\phi(y)=0$ implies $T^*\phi(y)=0$, there exists a scalar $c$ so that   $T^*\phi=c\phi=\pm \phi$.
Therefore $\phi(Tx)=\pm 1$. Strict convexity implies that $Tx=\pm x$.
\end{proof}

Summing up, we may note that on a separable transitive space, $Tx=\pm x$ if and only if $T(H_0)=H_0$. Transitivity is however not needed for the next result:

\begin{theo} \lb{fixedx0-v2}
Let $X$ be an almost transitive  real  Banach space. Suppose that for some $x_0 \in S_X$ supported by $\phi$, 
${\rm Stab}_{x_0}(X)$ acts 
almost  transitively on $S_{\Ker \phi}$. Then $X$ is isometric to a Hilbert space.
\end{theo}
\begin{proof}
By Theorem~\ref{hyperplane}, it is enough to prove that the hyperplane 
$\Ker \phi$ is 1-complemented in $X$, that is, that the  projection
$P(z)\DEF z - \phi(z)x_0$
has norm one.

Fix any $z\in S_X$ and let $\al= \phi(z)$. Since ${\rm Stab}_{x_0}(X)$ acts 
almost  transitively on $S_{\Ker \phi}$, for every $\ep>0$ there exists an isometry $T_\ep\in {\rm Stab}_{x_0}(X)$ so that 
$\|T_\ep(z-\al x_0)-(-(z-\al x_0))\|\le \ep$. Hence
$$\|T_\ep(z)-(2\al x_0-z)\|\le \ep.$$
Since $\ep>0$ is arbitrary we get that $\|z-2\al x_0\|=\|z\|=1$.
Thus
\begin{equation*}
\|P(z)\|=\|z-\al x_0\|=\Big\|\frac12 \big(z + (z-2\al x_0)\big)\Big\| \le \frac12\big(\|z\| + \|z-2\al x_0\|\big)\le 1,
\end{equation*}
which ends the proof.
\end{proof}

We finish this section with an observation that Theorem~\ref{fixedx0-v2} implies in particular that for any $1\le p<\infty$, $p \neq 2$, and for all 
$x_0\in  S_{L_p}$, the group ${\rm Stab}_{x_0}(L_p)$ does not act
almost  transitively on $S_{\Ker \phi}$. This is easy to see directly in the case when $x_0(t)=1$ for all $t\in [0,1]$. 

\begin{lemma}\label{lemmastabLp} If $1 \leq p <\infty, p \neq 2$, then
 ${\rm Stab}_1(L_p)$ does not act almost transitively on 
 $M=\{f: \int_0^1 f=0\}$.
\end{lemma}
 
\begin{proof}
Every isometry in ${\rm Stab}_1(L_p)$ is of the form $T(f)=f \circ\s$ where  $\s$ is a measure preserving automorphism of $[0,1]$. 
Consider $h=1_{[0,1/2)}-1_{[1/2,1]}\in S_M$. Then for any $T\in{\rm Stab}_1(L_p)$ the measure of the support of $Th$ is also equal to $1$. Thus, if  $f$ is any function in $S_M$ whose support has measure $\mu<1$, we have
$$\|Th-f\|_p\ge \Big(\int_{[0,1]\setminus \supp(f)} 1\Big)^{1/p}=(1-\mu)^{1/p}.$$
Therefore 
${\rm Stab}_1(L_p)$  does not act almost transitively on $S_M$. \end{proof}

\begin{small}
\noindent{\em Comments:}
\begin{itemized}
\item[(1)]
Proposition~\ref{pin2infty} is also true when $X$ is  a subspace of the Schatten class $S^p(\ell_2)$, $1<p<\infty$, $p\ne 2$ or  of the  non-commutative
$L_p[0,1]$, $2<p<\infty$,  and when $X$ is a
stable Banach space  that admits a $C^2$-smooth  bump.
% (the class of \textit{stable} spaces was introduced in \cite{KM}).
However it is open whether it remains true  when $X$ is  a subspace of 
$L_p$ for $1<p<2$. In this case we only know that  every subspace of 
$X$ contains isomorphic (even almost isometric) copies of $\ell_2$,  see \cite{DR} for  details.
 \item[(2)]
Notice that Lemma \ref{lemmastabLp} is true also in the case when $p=1$, despite the fact that, as we mentioned above, Talponen \cite{T09} showed that $\Isom(L_1)$ acts almost transitively on $S_{M}\subset S_{L_1}$. Talponen also showed that in this case for all 
$T\in\Isom(L_1)$ we have $T(M)=M$, but, of course, the conclusion of Lemma~\ref{L:pmx} does not hold.
\item[(3)] By the way, since every hyperplane of $L_1$ is isomorphic to the whole space, the result of Talponen \cite{T09} mentioned above   provides another AT renorming of $L_1$, different from those described in  
Theorem~\ref{maxcapitalLp}.
\end{itemized}
\end{small}

%\newpage

\section{Maximal norms and unitarisable representations on spaces isomorphic to Hilbert spaces}\label{3}

In this section we focus on Mazur rotations problem on a space already known to be linearly isomorphic to the Hilbert space. The   results presented in this section are mainly from \cite{FR2}.
A way of understanding this concept is by considering the {\em $G$-invariant norms}
corresponding to a bounded subgroup $G\leqslant \GL(X)$ on a Banach space $X$.  
%Here $G$ is {\em bounded} if$\|G\|:=\sup_{T\in G}\|T\|<\infty$. 
In the language of representations, the question is to investigate the {\em invariant norms} for a representation of a group $\Gamma$,
i.e. the norms for which the representation induces an action of $\Gamma$ by isometries on $X$. 
Recall from Section 2 that if $G$ is bounded, then 
$$
\triple x=\sup_{T\in G}\|Tx\|
$$
defines an equivalent $G$-invariant norm on $X$, i.e., $G$ may be seen as a subgroup of ${\rm Isom}(X, \triple\cdot)$. Moreover, if $\|\cdot\|$ is uniformly convex,
then  so is $\triple\cdot$ (see, e.g.,   \cite[Lemma 1.1]{CS3}; or used more recently from another perspective, \cite[Proposition 2.3]{furman}. However, if $X$ is a Hilbert space
and $\norm$ is hilbertian,
i.e., induced by an inner product, then $\triple\cdot$ will not, in general, be hilbertian. 
The {\em unitarizability} problem therefore asks which bounded subgroups of $\GL(\ku H)$ admit invariant euclidean norms. 
It is a classical result of representation theory dating back to the beginning of the 20th century that
if $G$ is a bounded subgroup of $\GL(\mathbb{C}^n)$, then there is a $G$-invariant inner product, or equivalently 
a $G$-invariant euclidean norm. B. Sz.-Nagy \cite{Nagy} showed  that any bounded representation 
$\pi\colon \Z\til \GL(\ku H)$ is {\em unitarizable}, i.e., $\ku H$ admits an equivalent $\pi(\Z)$-invariant
inner product.   This was extended by M. Day \cite{day} and 
J. Dixmier \cite{dixmier} to any bounded representation of an amenable topological group, via averaging over an invariant mean.

 In the opposite direction, the first example of a non-unitarizable bounded representation
 of a (necessarily non-amenable) group in the Hilbert is due to L. Ehrenpreis and F. I. Mautner \cite{mautner}.
 %This was later replaced by any countable group 
 %containing the free group $\F_2$ (or equivalently $\F_\infty$), ASK CHRISTIAN.
 Since, by a result of A. J. Ol'\v{s}hanski\u\i{} \cite{olshanski}, 
 there are non-amenable countable groups which do not contain a copy of $\F_2$, it remains open whether the  result of Sz.-Nagy,
 Day and Dixmier admits a converse.

\begin{prob}[Dixmier's unitarizability problem] 
Suppose $\Gamma$ is a countable group all of whose bounded representations on $\ku H$ are unitarizable. Is $\Gamma$ amenable?
\end{prob}

In  \cite{FR2} Ferenczi and Rosendal investigate the relation of certain non-unitarizable representations on the Hilbert
with the notions of maximality, almost transitivity, or transitivity of norms, through the following problem: 

\begin{prob} [Ferenczi-Rosendal, 2017] Find a
non-unitarizable representation on the Hilbert space which admits an equivalent invariant maximal (resp. almost transitive, transitive) norm. 
\end {prob}

In the case of a positive answer,  a maximal (resp. AT, transitive) non-hilbertian norm
on the Hilbert would be obtained, and the Hilbert space would admit non-conjugate maximal norms (see Problem \ref{3problems}). In the last
case, there would exist a transitive, non-hilbertian norm, on the Hilbert, and therefore a negative answer to
the isometric version of Mazur rotations problem.

%To study Problem \ref{prob}, we shall b the structure of bounded groups containing the image
%of one specific widely studied non-unitarisable representation associated to actions on trees 
%(see, e.g., \cite{ ozawa, pisier, pytlic}). 

We focus here on a specific class of possibly non-unitarizable representations on the Hilbert, which first appeared  in \cite{pytlic}:
triangular representations on a direct sum of two copies of the Hilbert, where the diagonal
elements of the matrix are unitary and where the upper right element is called a {\em derivation}.

Precisely, suppose that $\lambda\colon \Gamma\til \ku U(\ku H)$
is a unitary representation. A {\em bounded derivation} associated to $\lambda$ is a uniformly bounded map
$d\colon \Gamma\til \ku B(\ku H)$ so that $d(gf)=\lambda(g)d(f)+d(g)\lambda(f)$ for all $g,f\in \Gamma$. This is simply equivalent to requiring that
$$
\lambda_d(g)=\begin{pmatrix} \lambda(g) & d(g) \\ 0 & \lambda(g) \end{pmatrix}
$$
defines a bounded representation of $\Gamma$ on $\ku H\oplus \ku H$. The representation $\lambda_d$ is unitarizable
exactly when $d$ is {\em inner}, i.e., $d(g)=\lambda(g)A-A\lambda(g)$ for some bounded linear operator $A$ on $\ku H$
(a classical result whose proof may be found, e.g., in \cite{FR2}).

Of course such a representation $\lambda_d$ cannot be transitive or almost transitive, since it leaves the first summand invariant. Citing \cite{FR2} this leads to the study of
``bounded groups $G\leqslant {\GL}(\ku H\oplus \ku H)$ 
containing $\lambda_d[\Gamma]$ for $\lambda$ and $d$ as above, which are potential examples of maximal non-unitarizable groups":

%\begin{thm}\label{introduction: bounded}
%Let $X=Y\oplus Z$ be separable reflexive and $G\leqslant {\rm GL}(X)$ a bounded subgroup %leaving $Y$ invariant. Assume that there are no closed linear $G$-invariant subspaces %$\{0\}\varsubsetneq W\varsubsetneq Y$ nor superspaces $Y\varsubsetneq W\varsubsetneq X$ %and there is no closed linear $G$-invariant complement of $Y$ in $X$.
%Then the mappings
%$$
%\begin{pmatrix} u & w \\ 0 & v \end{pmatrix}\mapsto u
%\quad\text{and}\quad
%\begin{pmatrix} u & w \\ 0 & v \end{pmatrix}\mapsto v
%$$
%are $\mathtt{sot}$-isomorphisms between $G$ and the respective images in ${\rm GL}(Y)$ and %${\rm GL}(Z)$.
%\end{thm}
 
% we apply  Theorem \ref{introduction: bounded} when $G \leqslant {\rm GL}(\ku H \oplus \ku %H)$ is a bounded subgroup leaving the first copy of $\ku H$ invariant and containing the image %$\lambda_d[\Gamma]$, where $\lambda$ is an irreducible unitary representation and $d$ is an %associated non-inner derivation. 

\begin{prop}[\cite{FR2}]\label{lip}
Suppose that $\lambda\colon \Gamma\til \ku U(\ku H)$ is a unitary representation of a group $\Gamma$ on a separable infinite-dimensional Hilbert space $\ku H$
and let $d$ be a bounded derivation associated to $\lambda$. Consider the assertions
\begin{enumerate}
\item There is an almost transitive bounded subgroup $G$ of
${\GL}(\ku H_1 \oplus \ku H_2)$ containing $\lambda_d[\Gamma]$.
\item There is a $\lambda_d[\Gamma]$-invariant norm on  $\ku H_1\oplus \ku H_2$ with moduli of convexity and smoothness of power type $2$.
\item There is a $\lambda_d[\Gamma]$-invariant norm on  $\ku H_1\oplus \ku H_2$
such that the $\ku H_1$-nearest point map 
$\ku H_1 \oplus \ku H_2 \til \ku H_1$ is well-defined and Lipschitz.
\item There is a homogeneous Lipschitz map $\psi: \ku H_2 \til \ku H_1$ such that
$d(a)=\lambda(a)\psi-\psi \lambda(a)$.
\item The group 
$\lambda_d[\Gamma_z]$ is unitarizable for $z$ outside of a Gauss null subset of $\ku H_2$,  where $\Gamma_z=\{a \in \Gamma : \lambda(a)(z)=z\}$.
\end{enumerate}
Then {\rm (1)} $\implies$ {\rm (2)} $\implies$ {\rm (3)} $\implies$ {\rm (4)} $\implies$ {\rm (5)}.
\end{prop}

\begin{proof}
The idea of the proof is as follows. From (1) one deduces that the $G$-invariant norm $\sup_{g \in G}\|gx\|_2$ (which has modulus of convexity of power type $2$)
is a multiple of any given $G$-invariant norm on $\ku H=\ku H_1 \oplus \ku H_2$. The same holds for the dual norm to the $G$-invariant norm
$\sup_{g \in G}\|g^* \phi\|_2$, defined on the dual, and this norm has modulus of smoothness of power type $2$. 
So there is a $G$-invariant (and in particular $\lambda_d[\Gamma]$-invariant) norm with both moduli of power type $2$.
The implication (2) $\implies$ (3) follows from classical  estimates relating the modulus of continuity of the nearest point to the moduli of convexity and smoothness,  which   appear in \cite{BL} as Theorem 2.8.
Since the $\ku H_1$-nearest point map $n: \ku H_1 \oplus \ku H_2 \rightarrow \ku H_1$ is equivariant under translation by any vector in $\ku H_1$ and under isometries
in $\lambda_d[\Gamma]$, it is given by the formula:
$n(x,y)=x+\psi(y)$. The map $\psi(y)=n(0,y)$ is Lipschitz and the identity $d(a)(x)=\lambda(a)\psi(x)-\psi(\lambda(a)x)$
follows from the relation $n(T(x,y))=T(n(x,y))$ for any $T=\lambda_d(a)$.

(4) $\implies$ (5) Outside of a Gauss null set the map $\psi$ is G\^ateaux differentiable (\cite[Theorem 6.42]{BL}). Deriv\-a\-ting the relation
above for $a \in \Gamma_z$, $\psi'(z)$ witnesses that $d(a)$ is a linear derivation for the group 
$\lambda_d[\Gamma_z]$.
\end{proof}

It may be interesting to note here that geometric properties of general Banach spaces (such as uniform convexity or smoothness) are relevant even to the seemingly trivial case of a Hilbert space. For example, choosing to see a bounded group on $\ku H$   as an isometric group on some non-hilbertian renorming $X$ of $\ku H$ allows to use relations of the nearest point map with convexity or smoothness of the norm of $X$.

\begin{cor}\label{introduction: hilbert}
Suppose that $\lambda\colon \Gamma\til \ku U(\ku H)$ is a unitary representation of a group $\Gamma$
on a separable  Hilbert space $\ku H$ and $d\colon \Gamma\til \ku B(\ku H)$ is an associated non-inner bounded derivation.
Suppose that $G \leqslant {\GL}(\ku H \oplus \ku H)$ is a bounded   almost transitive subgroup containing $\lambda_d[\Gamma]$. Then 
there is a homogeneous Lipschitz non-linear map $\psi\colon \ku H\til \ku H$ defining the derivation
by $d(a)=\lambda(a)\psi-\psi\lambda(a)$. %but $\psi$ may not be chosen to be linear.
\end{cor}

The authors of \cite{FR2} call {\em Lipschitz inner} a bounded derivation of the form 
$d(g)=[L,g]$, for $L$ Lipschitz homogeneous on $\mathcal H$, and ask the following natural question (\cite{FR2}, end of Section 3):

%The authors conclude this part with the natural question:

\begin{prob}  Does there exist a Lipschitz inner derivation on $\ku H$ which is not inner?
\end{prob}

It is unclear whether differentiability techniques may be used to obtain that every Lipschitz inner derivation is inner.
Those techniques usually do not have  any kind of invariance or equivariance with respect to the action of the
isometry group and this seems to be an unsurmountable problem.

\

\begin{small}

\noindent{\em Comments:}
\begin{itemized}
\item[(1)]
The survey  \cite{pisier} by G. Pisier and also \cite{monod1,monod2} contain the present state of affairs on Dixmier's problem.
\item[(2)]
 F. Cabello S\'anchez \cite{CS10variaciones} gives some partial restrictions on almost transitive renormings of Hilbert spaces.
Such  renormings must be twice G\^ateaux differentiable everywhere apart from zero, and the duality mapping must be
G\^ateaux differentiable everywhere apart from zero. For recent results regarding AT or transitivity of certain ``Schatten restricted" renormings of the Hilbert, see \cite{miglioli}. 
\item[(3)]
It is a classical geometric problem in Banach space theory
whether a (necessarily superreflexive) space admitting an equivalent norm with modulus of convexity of power
type $p$, and another with modulus of smoothness of power type $q$, must admit an equivalent norm with both properties.
Although such results hold for the LUR property, through Baire category methods on the set of equivalent norms, \cite[Section II.4]{DGZ}, the same method
does not apply to uniformly convex norms. It was noticed by C. Finet \cite{finet} that this would hold if every superreflexive space admitted
an equivalent almost transitive norm (indeed every almost transitive norm on a superreflexive space must have
modulus of convexity of optimal power type).%\m{Thanks}
But this hope was shattered by the example of \cite{FR} and later
by those of \cite{DR}.
When $p=q=2$ the question becomes trivial because of Kwapie\'n's theorem \cite{kwapien}.
However, given a bounded group $G$ of automorphisms on the space, the version of this problem for $G$-invariant norms remains open,
even for $p=q=2$ (note that on the Hilbert it is only relevant for non unitarizable groups):
\end{itemized}
\end{small}
\begin{prob}  Assume $G$ is a bounded group of automorphisms on the Hilbert space $\ku H$, and that there exist $G$-invariant norms on $\ku H$ with modulus of convexity
(resp. modulus of smoothness) 
of power type $2$. Must there exist a $G$-invariant norm on $\ku H$ with these two properties?
\end{prob}

\section{Multidimensional Mazur problem}\label{4}

As a general principle, we wish to  identify  properties of Banach spaces which are stronger than
transitivity, satisfied by Hilbert spaces, and for which however there exist non-hilbertian non-separable examples. 
Any positive solution to the Mazur rotations problem would need to solve the rotations problem associated to this stronger version of transitivity
as a first step. The direction explored for this in this section is {\em  multidimensionality}, and its results are mainly from the recent paper  \cite{FLMT}.

\subsection{Ultrahomogeneity}\label{subsec:ultra}

{\em Ultrahomogeneity} (or {\em ultratransitivity}) of a Banach space is the multidimensional version of the transitivity property.
The term ``ultrahomogeneity" is closer to tradition in the Fra\"iss\'e theory and this explains the choice of this term
in the paper \cite{FLMT}, to which we adhere.

It is undisputable that all $1$-dimensional spaces are mutually isometric. In higher dimensions, however,
a global isometry on a space can only send a finite dimensional subspace onto another if those
were isometric to begin with.

\begin{definition}[Ultrahomogeneity]
 A Banach space
  $X$ is said to be {ultrahomogeneous} (UH) when for every finite dimensional subspace
  $E\subset X$ every isometric embedding $E\to X$ extends to a global (surjective) isometry of $X$.
  %and every isometric embedding $\phi:E\to X$ there is a linear isometry $g\in \mr{Isom}(X)$ such that $g\rest E=\phi$;
%this means 
%the canonical action $${\rm Isom}(X) \lop {\rm Emb}(E,X)$$ is transitive.
\end{definition}

Less clearly, $X$ is UH if for every finite dimensional $E\subset X$ the
canonical action ${\rm Isom}(X) \lop {\rm Emb}(E,X)$ is transitive (see Section~1.2 for the definition of ${\rm Emb}(E,X)$).

Note that any ultrahomogeneous space (or norm) is in particular transitive. As a consequence of the existence of orthogonal complements in Hilbert spaces:

\begin{fact} Hilbert spaces are ultrahomogeneous.
\end{fact}

\begin{prob}[Multidimensional Mazur problem]\label{prob:multiMazur} 
Is every separable in\-fi\-ni\-te dimensional ultrahomogeneous Banach space isometric (or isomorphic) to the Hilbert space?
\end{prob}

Similarly  as for the one-dimensional problem, this leads quite naturally to two separate questions namely:  Is every separable UH Banach space isomorphic to a Hilbert space? Is every UH renorming  of a Hilbert space Euclidean?

Is there any other (nonseparable) UH space in sight? 
Yes, ultrapowers of the Gurariy space or of  $L_p$-spaces for appropriate values of $p$, with respect to CI ultrafilters. See below.

\
\begin{small}

\noindent{\em Comments:}
\begin{itemized}
 \item[ ]
The two-dimensional part of the UH property (that any isometric embedding of a two-dimensional subspace extends to a surjective isometry) is not to be confused with the notion of $2$-transitivity (whenever $x,y,x',y'$ belong to the sphere, and $d(x,y)=d(x',y')$, then there exists a surjective isometry sending $x$ to $x'$ and $y$ to $y'$). 
The second one is much stronger and already implies that the space is isometrically hilbertian (no separability needed): Ficken \cite{ficken44} proved that
  if for all $x,y\in S_X$ there exists $T\in \Isom(X)$ with $T(x)=y$ and $T(y)=x$, then $X$ is isometric to a Hilbert space,
  see also    \cite[Condition 2.8]{Amir}.
\end{itemized}
\end{small}

\subsection{Approximate ultrahomogeneity}\label{subsec:approx} Let us   introduce the following
``approximate'' version of UH, taken from  \cite{FLMT}: 

%\subsection{Approximate ultrahomogeneity}
%The following notion, which is obviously stronger than almost transitivity but weaker than ultrahomogeneity, is defined in \cite{FLMT}.

\begin{definition} %[Approximate Ultrahomogeneity]
A Banach space  $X$  is called approximately ultrahomogeneous (AUH)  when for every finite dimensional subspace $E$ of $X$,
   every isometric embedding $u:E\to X$ and every $\vep>0$ there is an isometry $U$ of $X$ such that $\|u|_E-U\|<\vep$.
\end{definition}

Thus $X$ being AUH exactly means that the canonical action { ${\rm Isom}(X) \lop {\rm Emb}(E,X)$} is {\em almost transitive} (i.e. has dense orbits),
where ${\rm Emb}(E,X)$ is equipped with the metric induced by the operator norm; informally, this means that any partial isometry between finite dimensional subspaces can  be well approximated by a global isometry.

The following sums up the known examples of separable, non hilbertian,  AUH spaces.
   
\begin{theorem}
The following spaces are AUH, but not UH:
\begin{itemize}
\item[(a)] The Gurarij space $\mathcal G$.% (Kubis-Solecki 2013, \cite{kubissolecki}).
\item[(b)] $L_p$, for $p\neq 2,4,6,8,\dots$ % (Lusky 1978, \cite{luss}).
\end{itemize}
\end{theorem}
The AUH character of $\mathcal G$ is a relatively recent result
 by Kubi\'s and Solecki  \cite{kubissolecki}. The Gurariy space is the only universal, separable AUH Banach space.
The part concerning $L_p$ spaces was essentially established by Lusky in the late 1970s \cite{luss} elaborating on the {\em Plotkin/Rudin equimeasurability theorem}.  It is clear that these spaces cannot be UH because they are not even transitive.
It is a remarkable fact that the $L_p$ spaces, for $p=4,6,\dots$ fail to be AUH. This
follows from   work of  Randrianantoanina \cite{beata} who, as part of an answer to a question of H.P. Rosenthal \cite{Ro}, showed 
that those spaces contain isometric copies of certain finite dimensional spaces with very different projection constants; %  which can be well complemented and (arbitrarily) badly complemented copies on $L_p$. 
see details in \cite{FLMT}. This is quite surprising since  for $p\in(1,\infty)$ different from 2, the groups $\operatorname{Isom}(L_p)$ are all topologically isomorphic to each other, including $p$ even, both in the SOT and in the norm topology. However, their canonical actions on $\operatorname{Emb}(E,L_p)$ turn out to have very different properties, depending on whether or not $p$ is an even integer.

The Garbuli\'nska space $\mathcal K$ 
described in Section~\ref{sec:classical-isom} provides a more  ``ca\-non\-ical'' example of an  AT space which is not AUH. This can be seen as follows. Every Banach space with a skeleton (in particular, a finite dimensional one) is isometric to a 1-complemented subspace of
 $\mathcal K$. This applies to $\ell_\infty^{2^n}$ and $\ell_1^n$. But $\ell_\infty^{2^n}$ contains an isometric copy of $\ell_1^n$ whose projection constant is large (let us be foolhardy: it is exactly $\frac{2m+1}{2^{2m}}\binom{2m}{m}$, where $m$ is the integer part of ${1\over 2}(n-1)$, proved by Gr\"unbaum in 1960, \cite{G60}). Thus $\mathcal K$ contains well- and bad-complemented subspaces isometric to $\ell_1^n$ so that it canot be AUH, and neither can its ultrapowers.

 As we already mentioned, there exist non-separable ultrahomogeneous spaces. A method of finding them used in \cite[Chapters 3 and 4]{ACCGM-LN}, and then in \cite{FLMT}, has been  %\m{B: note to self: look at it again}
to investigate weaker forms of transitivity of separable spaces, with the objective of then taking ultrapowers.
What catches us off-guard is that the AUH of a Banach space does not automatically imply UH of its ultrapowers. 
The reason for this is that, in general, an isometric embedding $u: E\to [X_i]_\ultrafilter$ can arise from a family of $\eps_i$-isometric embeddings $u_i: E\to X_i$ with $\eps_i\to 0$ along $\mathcal U$.  

Nevertheless, the Gurariy space, being separable and of almost universal disposition, has the following ``perturbed'' version of UH that is much easier to establish than AUH and was known to Gurariy himself:

\begin{lemme}
Let $u: E\to F$ be an $\delta$-isometry acting between two finite dimensional subspaces of $\mathcal G$. Then, for every $\eps>\delta$ there is a surjective $\eps$-isometry $U$ of $\mathcal G$ extending $u$.
\end{lemme}

Curiously enough, no isometry {\em sensu stricto} is involved in the preceding statement. 
As a consequence we have the following result \cite[Proposition 4.16]{ACCGM-LN}:

\begin{prop}\label{accgm-lnm}
Ultrapowers of the Gurariy space built on countably incomplete ultrafilters are ultrahomogeneous.
\end{prop}

The density character of any such space is at least the continuum; we do not know if there are examples whose density character is $\aleph_1$; see Section~\ref{sec:classical-isom} {\em and} the comments around \cite[Proposition 4.2]{cgk}.

And what about ultrapowers of $L_p$? Keep reading.

\subsection{Fra\"iss\'e Banach spaces} One of the main technicalities %\m{F: of the Fra\"iss\'e definition  for Banach spaces $\to$ }
of the definition of a Fra\"iss\'e  Banach space from
\cite{FLMT} is that it is expressed in terms of the canonical
actions of the linear isometry group, not only on the spaces ${\Emb}(E,X)$ of isometric embeddings, but also on ${\Emb}_\de(E,X)$, the class of $\de$-isometric embeddings, which is equipped with the distance induced by the norm. As the reader may guess, the canonical action  
 ${\rm Isom}(X) \lop {\rm Emb}_\de(E,X)$ is defined by $(g,T) \mapsto g \circ T$. Also,
 the
 action of a subgroup $G$ of ${\rm Isom}(X)$  on $ {\rm Emb}_\de(E,X)$ is said to be {\em $\vep$-transitive} 
 if for any $T,U \in {\rm Emb}_\de(E,X)$,
there exists $g \in G$ such that $\|g \circ T-U\| \leq \vep$.

Following a terminology inspired by the Fra\"iss\'e theory (but without using the abstract setting of model theory which is common in the general
Fra\"iss\'e theory), given a Banach space $X$, we denote by $\age(X)$ the set of all finite dimensional subspaces of $X$,
and by $\age_k(X)$ the set of its $k$-dimensional subspaces.
Our presentation of the results of \cite{FLMT} is slightly modified  to point out the role of the dimension.
%$k$ of finite dimensional subspaces of $X$.
  
%We say that the
% action of a subgroup $G$ of ${\rm Isom}(X)$  on $ {\rm Emb}_\de(E,X)$ is {\em $\vep$-transitive} 
% if for any $T,U \in {\rm Emb}_\de(E,X)$,
%there exists $g \in G$ such that $\|g \circ T-U\| \leq \vep$.

%While $\de$-isometric embeddings were already considered by M. Lupini \cite{Lup}, in a general theory of stability including,
%for example, operator spaces and systems, in this paper one of our objectives is to obtain finer results based on weaker (and/or more precise)
%properties of homogeneity for structures, in such a way that $L_p$ spaces are included in the classes we consider. With these examples in mind,
%we concentrate on the case of the Banach spaces, and develop a theory which may be specific to the Banach space setting.

% In particular, although our results should be extendable to the quasi-Banach setting and %the case of $L_p$ spaces for $0<p<1$, we shall not consider that situation.

\begin{definition} [Ferenczi, L\'opez-Abad, Mbombo, Todorcevic \cite{FLMT}]  Let $k \in \N$. A Banach space
$X$ is {$k$-Fra\"iss\'e}
 if and only if for every $\vep>0$ there is $\de=\de_k(\vep)>0$ such that for every $E \in \age_k(X)$, 
 the action ${\rm Isom}(X) \lop {\rm Emb}_\de(E,X)$ is $\vep$-transitive. 
A Banach space
$X$ is {Fra\"iss\'e}
 if and only if it is $k$-Fra\"iss\'e for every $k\in \N$. 
\end{definition}

Since isometric embeddings are $\delta$-isometric for any $\delta>0$,  Fra\"iss\'e $\Rightarrow$ (AUH). We pass to an important characterization   of the
Fra\"i\-ss\'e property indicating that the possibility of choosing $\delta$ uniformly on subspaces of dimension $k$ is related to the closedness of 
${\rm Age}_k(X)$ in the Ba\-nach-Ma\-zur compactum.

\begin{defin}  A space $X$ is weak $k$-Fra\"iss\'e if and only if for every $E \in \age_k(X)$ and every $\vep>0$, there is $\de=\de_E(\vep)>0$
 such that 
 the action ${\rm Isom}(X) \lop {\rm Emb}_\de(E,X)$ is $\vep$-transitive. 
\end{defin}

The following is proved in \cite[proof of Theorem 2.12]{FLMT}: %(although %not stated exactly in this manner).

\begin{lemma}\label{lem:eqweak}
The following are equivalent for $X$ Banach and $k \in \N$:
\begin{itemize}
\item[(1)] $X$ is  $k$-Fra\"iss\'e,
\item[(2)] $X$ is weak $k$-Fra\"iss\'e and ${\rm Age}_k(X)$ is compact in the Banach-Mazur (pseudo) distance.  
\end{itemize}
 \end{lemma}

%\begin{proof} See the proof of Theorem 2.12 in \cite{FLMT}. \end{proof}
 
 And therefore
 
 \begin{prop}\label{prop:eqweak}
  For a Banach space $X$ the following are equivalent:
  \begin{itemize}
    \item[(1)] $X$ is Fra\"iss\'e,
   \item[(2)] $X$ is weak Fra\"iss\'e and for all $k \in \N$, ${\rm Age}_k(X)$ is compact in the Banach-Mazur (pseudo) distance. 
  \end{itemize}
 \end{prop}

\begin{small}

\noindent{\em Comments:}
\begin{itemized}
 \item [ ]
Given a (hereditary) class { $\mathscr F$} of finite (or sometimes finitely generated) structures, the
{\em Fra\"iss\'e theory} (Fra\"iss\'e 1954, \cite{Fra}) investigates the existence of a countable structure {${\mathcal A}$},
universal for ${\mathscr F}$ and {\em ultrahomogeneous}
(any  isomorphism between finite substructures extends to a global automorphism of ${\mathcal A}$).  The
``Fra\"iss\'e correspondence" shows that this is equivalent to certain ``amalgamation properties" of ${\mathscr F}$. 
 In this case ${\mathcal A}$ is {unique} up to an isomorphism and called the { Fra\"iss\'e limit} of ${\mathscr F}$. Analogies of this situation with the ultrahomogeneity properties of Banach spaces considered in their paper led to the use of the Fra\"iss\'e terminology in \cite{FLMT}.
 \end{itemized}
 \end{small}
 
\subsection{Examples of Fra\"iss\'e Banach spaces}\label{sec:exF}

As expected, the list of usual suspects provides examples of Fra\"iss\'e   spaces:

%\begin{prop} The following hold:
%\begin{itemize}
%\item[(1)] Hilbert spaces are Fra\"iss\'e ($\vep=\delta$,  \cite[Example 2.4]{FLMT});  
%\item[(2)] the Gurarij space is Fra\"iss\'e ($\vep=2\delta$, \cite[Example 2.5]{FLMT}) 
%\item[(3)] $L_p(0,1)$ is {not} Fra\"iss\'e  for $p=4,6,8,\dots$ since not AUH.
%\end{itemize} 
%\end{prop}

%One of the main results of \cite{FLMT} is:

\begin{thm}\label{big}
The following Banach spaces are Fra\"iss\'e:
\begin{itemize}
\item[(a)] Hilbert spaces (with $\vep=\delta$),%  \cite[Example 2.4]{FLMT});  
\item[(b)] the Gurariy space (with $\vep=2\delta$),% \cite[Example 2.5]{FLMT}) 
\item[(c)] The spaces $L_p$ for finite $p\neq 4,6,8,\dots$
\end{itemize}
% (Ferenczi, Lopez-Abad, Mbombo, Todorcevic)
%The spaces $L_p[0,1]$  for $p\neq 4,6,8,\dots$ are Fra\"iss\'e.
\noindent However $L_p$ is {not} Fra\"iss\'e  for $p=4,6,8,\dots$ since is not AUH.
\end{thm}

Part (a) is very easy: it consists in showing that every $\delta$-isometric embedding 
between finite dimensional Hilbert spaces is at distance $\delta$ from a true isometric embedding, see \cite[Example 2.4]{FLMT} for details. Part (b) is due to Kubi\'s and Solecki \cite[Theorem 1.1]{kubissolecki}. Part (c) is a recent result by Ferenczi, L\'opez-Abad, Mbombo and Todorcevic \cite[Theorem 4.1]{FLMT}.

\
\begin{small}

\noindent{\em Comments:}
\begin{itemized}
\item[(1)] Citing Lusky \cite{luss},
{\em ``We show that a certain homogeneity property holds
for $L_p; p \neq 4,6,8,\ldots$ which is similar to a corresponding property of the Gurariy space...''}
The Fra\"iss\'e Banach space definition gives a more precise  meaning to this similarity.

\item[(2)] The proof of Theorem \ref{big}(c) is quite technical and will not be presented here.
It is based on proving an {\em approximate equi\-measurabili\-ty principle},
a continuous statement extending the classical equi\-mea\-surability principle of  Plotkin and  Rudin,
see  \cite[Section 4.2]{FLMT}.
%, and that, it seems, can be interpreted as the Ryll-Nardzewski Theorem for the $\aleph_0$-categorical theory of those $L_p$ spaces: 
%supposing that $\mu,\nu$ are Borel measures on $\R^n$ for which the coordinate  functions  $x_j$ are $p$-integrable.
%If $\widehat{\mu}^{(p)}(a):=\int |1+ <a,z>|^p d\mu(x)=\int |1+ <a,z>|^p d\nu(x)=:\widehat{\nu}^{(p)}(a)$
%for all $a \in \R^n$ are "close" then the measures $\mu$ and $\nu$ are "close"(in contrast with
%the exact statement in the classical equimeasurability theorem - obtaining equality of $\mu$ and $\nu$ from equality of $\widehat{\mu}^p$ and $\widehat{\nu}^p$),  
  This result implies a local statement about extension of almost isometric embeddings which is equivalent to
  the weak Fra\"iss\'e property for $L_p$. 
  The other ingredient is the classical fact from the theory of $L_p$-spaces, 
  that ${\rm Age}_k(L_p)$ is compact in the Banach-Mazur distance.
  One then concludes the proof by Proposition \ref{prop:eqweak}.
  
 \item[(3)] An optimal estimative of the values of $\delta(k,\epsilon)$ appearing in the Fra\"is\-s\'e property for the space $L_p$
  remains to be computed. In particular, it is unclear whether $\delta$ could be chosen uniformly in $k$, see the next item.

\item[(4)]
The estimates obtained on $\delta$ in the cases of the Hilbert space and the Gurariy space 
witness  that $\delta$ may be chosen independently of the dimension of the subspace $E$. This leads to the following definition, see \cite[p. 5]{FLMT}, as well as \cite{Lup} in a much more general context:
a Banach space $X$ is {\em stable Fra\"iss\'e} 
 if for every $\vep>0$ there is $\de=\de(\vep)>0$ such that for every $E \in \age(X)$, 
 the action ${\rm Isom}(X) \lop {\rm Emb}_\de(E,X)$ is $\vep$-transitive. %\m{Comment added following an observation of the referee}
 Please note that the meaning of the adjective ``stable'' here has nothing to do with stable spaces in the sense of Krivine-Maurey.  The following natural question is open:   
 \end{itemized}
 \end{small}

%\begin{prob} Are the Gurarij and the Hilbert the only separable stable Fra\"iss\'e spaces? \end{prob}

\begin{prob}\label{pr:LpstableFraisse} Are the spaces $L_p$, for finite $p\neq 4,6,8\dots$, stable Fra\"iss\'e? \end{prob}

\subsection{Embeddings and isometries between Fra\"iss\'e spaces}

The following properties of Fra\"iss\'e Banach  spaces may  be thought
of as  natural counterparts to their ``exact" equivalent statements in the Fra\"is\-s\'e theory,
relating an ultrahomogeneous countable structure to its finite parts. 

Recall that a space $Y$ is finitely representable in $X$ if for any finite dimensional subspace $E$ of $Y$ and any $\vep>0$, there exists a finite dimensional subspace $F$ of $X$ such that $d_{\rm BM}(E,F) < 1+\vep$. This is a basic notion
of local theory of Banach spaces, which aims to compare the finite dimensional structures of spaces
``up to arbitrarily small perturbation".

\begin{prop} Assume $X$ is Fra\"iss\'e, and that $Y$ is separable.
Then the following are equivalent:
\begin{itemize}
\item[(1)] 
 $Y$ is finitely representable in $X$.
\item[(2)] Every finite dimensional subspace of $Y$ embeds isometrically in $X$.
 \item[(3)] $Y$ embeds isometrically in $X$.
\end{itemize}
\end{prop}

Therefore embeddings into Fra\"iss\'e spaces are exactly prescribed by the natural order relation between the
respective local structures.
In particular, by the Dvoretzky's theorem about finite representability of the Hilbert in infinite dimensional Banach spaces
(cf.  \cite{GM}), all Fra\"iss\'e spaces must contain an isometric copy of the Hilbert space $\ku H$:

\begin{prop}
The Hilbert space is the minimal separable Fra\"is\-s\'e spa\-ce.
\end{prop}

%Let
%\begin{itemize}
%\item { $Age(X)$}=the set of finite dimensional subspaces of $X$,
%and
%\item for ${\cal F}, {\cal G}$ classes of finite dimensional spaces, 
%{ ${\cal F} \equiv {\cal G}$} mean that any element of ${\cal F}$ has an isometric %copy in ${\cal G}$ and vice-versa.
%\end{itemize}

Let
for ${\mathscr F}, {\mathscr G}$ classes of finite dimensional spaces, 
${\mathscr F} \equiv {\mathscr G}$ mean that any element of ${\mathscr F}$ has an isometric copy in ${\mathscr G}$ and conversely. By means of a
back-and-forth argument it is  proven in \cite[Proposition 2.22 and Theorem 2.19]{FLMT} that separable AUH (resp. Fra\"iss\'e) spaces are uniquely isometrically determined (among spaces with the same property)
by their age modulo $\equiv$ (resp. by their local structure). Precisely:

\begin{prop}\label{unique}  Assume $X$ and $Y$ are separable AUH spaces. Then the following are equivalent:
\begin{itemize}
\item[(1)] $\operatorname{Age}(X) \equiv \operatorname{Age}(Y)$, 
 \item[(2)] $X$ and $Y$ are isometric.
 \end{itemize}
 If furthermore $X$ and $Y$ are assumed to be Fra\"iss\'e spaces, then 
 $(1)$\&$(2)$ are also equivalent to
\begin{itemize}
\item[(3)] 
 $X$ is finitely representable in $Y$ and vice-versa.
\end{itemize}
\end{prop}

A consequence of Proposition \ref{unique} is that any separable Fra\"iss\'e space which does not have non-trivial cotype must be isometric to the Gurariy. Indeed, $\ell_\infty$ is finitely representable in such a space, and therefore condition (3) of Proposition \ref{unique} may be applied.

\subsection{Internal characterizations: amalgamation}
 In \cite{FLMT} are also obtained
 internal characterizations of those classes of finite dimensional spa\-ces  cor\-res\-ponding to % are $\equiv$ to 
 the age of some Fra\"iss\'e space (``amalgamation properties").
%The Fra\"iss\'e amalgamation property is as follows: 

\begin{defin} A class ${\mathscr F}$ of finite dimensional spaces has the Fra\"iss\'e amalgamation property
%For  the class $\age(L_p(0,1))$ there is not known full pushout; instead,  for $p\notin 2\N$ there is a 
%restricted version stating that for every $k\in \N$ and $\vep>0$ there is $\de>0$ such that 
if whenever  $E, F, G \in {\mathscr F}$  with $\dim E=k$, and $\ga\in {\Emb}_\de(E,F)$, $\eta\in {\Emb}_\de(E,G)$, 
there are $K \in {\mathscr F}$ and isometric embeddings $i: F \to K$ and $j:G \to K$ such that 
$\|i\circ \ga -j \circ \eta\|\le \vep$. \end{defin} 
%This is exactly, by means of the  Fra\"iss\'e correspondence,
%the Fra\"iss\'e property of $L_p(0,1)$.

It is not hard to check that the age of a Fra\"iss\'e Banach space must satisfy the Fra\"iss\'e amalgamation property.
Conversely and more importantly, the amalgamation property
is equivalent to the existence of an associated Fra\"iss\'e space $X$, 
which, by Proposition \ref{unique}, in the separable case, is uni\-que\-ly determined.

\begin{defin}
For a class ${\mathscr F}$ with the Fra\"iss\'e amalgamation property, there exists an isometrically unique  separable Fra\"iss\'e space $X$
such  that  ${\rm Age}(X) \equiv {\mathcal F}$. In this case it is said that $X$ is the {\em Fra\"iss\'e limit of
${\mathscr F}$}.
\end{defin}

\begin{quest} Are there other examples of amalgamation classes, apart from the classes of finite dimensional subspaces of
 $L_p$ for $p \neq 4,6,8\dots$, or the class of all finite dimensional normed spaces? 
\end{quest}

It may be that a hypothetical new separable Fra\"iss\'e space will appear not as a ``preexisting space" such as the $L_p$'s
but rather as a ``new space" defined as the limit of a new amalgamation class.

\subsection{Fra\"iss\'e is an ultraproperty}

In the same spirit as the relation between AT and transitivity in Section \ref{sec:intro}, there exist characterizations of the Fra\"iss\'e property through ultrapowers.
This point of view allows for formulations of the Fra\"iss\'e property without use of epsilontics. 

Given an ultrafilter $\mathcal U$,  denote by
$(\iso(X))_\ultrafilter$ the subgroup of isometries  of $X_\ultrafilter$ of the form $(T_i)_\ultrafilter$, where each $T_i \in {\iso}(X)$. Note that $(T_i)_\ultrafilter$ is (correctly) defined on $X_\mathcal U$ by $(T_i)_\ultrafilter[(x_i)]= [(Tx_i)]$.

We state here $k$-dimensional versions of some general properties proved  in \cite{FLMT}. 
%Their validity is obvious
%from the proofs of \cite{FLMT}.

\begin{prop}\label{prop111}
Let $\ultrafilter$ be a free ultrafilter on $\mathbb N$. 
For a Banach space $X$ and $k\geq 1$ the following are equivalent:

\begin{enumerate}
\item $X$ is weak $k$-Fra\"iss\'e (resp. $X$ is $k$-Fra\"iss\'e).
\item  For  every $E\in \age_k(X)$ 
(resp. for every $E\in \age_k(X_\ultrafilter)$), the action  $({\iso}(X))_\ultrafilter\acts {\rm Emb}(E,X_\ultrafilter)$ is almost transitive. 

\item  For  every $E\in \age_k(X)$ (resp. for every $E \in \age_k(X_\ultrafilter)$),
the action $({\iso}(X))_\ultrafilter\acts {\rm Emb}(E,X_\ultrafilter)$ is transitive. 
%\item For  every $\vep>0$ and  every $E\in \age_k(X)$ there  $\de>0$ such that    $(\iso(X))_\ultrafilter \acts \Emb_\de(E, X_\ultrafilter)$ is $\vep$-transitive. 
\end{enumerate} 
\end{prop}

%This is enhanced as follows in the Fra\"iss\'e case:
%\begin{prop}\label{prop112} The following are equivalent for $X$ Banach and $k \in \N$:
%\begin{enumerate}
%\item $X$ is $k$-Fra\"iss\'e.

%\item  $X_\ultrafilter$ is Fra\"iss\'e    and $(\iso(X))_\ultrafilter$ is SOT-dense in $\iso(X_\ultrafilter)$
%(equiv. the product  of the discrete topology).
%\item  For  every $E\in \age_k(X_\ultrafilter)$ the action  $(\iso(X))_\ultrafilter\acts {\rm Emb}(E,X_\ultrafilter)$ is almost transitive. 

%\item  For  every $E\in \age_k(X_\ultrafilter)$ the action $(\iso(X))_\ultrafilter\acts {\rm Emb}(E,X_\ultrafilter)$ is transitive. 
%\end{enumerate}
%\end{prop}

Note the difference between finite dimensional subspaces of $X$ and finite dimensional
subspaces of $X_\ultrafilter$.
By classical results of local theory and ultraproducts, the elements of $\age(X_\ultrafilter)$ are exactly those belonging to the
 closure of the set $\age(X)$ with respect to the Banach-Mazur distance. As an illustration, every finite dimensional
subspace of $L_p$ is a limit (in the Banach-Mazur distance) of a sequence of finite dimensional subspaces of $\ell_p$; $\age(L_p)$ is closed but $\age(\ell_p)$ is not.

Several equivalent characterizations of the Fra\"iss\'e property appear in \cite{FLMT} and follow essentially from Proposition \ref{prop111}. Informally and under some restrictions, we see that UH, %\m{B: replaced the ultrahomogeneous} 
AUH and the Fra\"iss\'e property induced by isometries on the space, become indistinguishable in its ultrapowers.

\begin{prop}\label{prop113}
The following are equivalent for a Banach space $X$:
\begin{enumerate}
 \item $X$ is Fra\"iss\'e.
 \item  The action  $({\iso}(X))_\ultrafilter \acts {\rm Emb}(E,X_\ultrafilter)$ is almost transitive  for  every $E\in \age(X_\ultrafilter)$. 

\item The action  $({\iso}(X))_\ultrafilter \acts {\rm Emb}(E,X_\ultrafilter)$ is transitive  for  every $E\in \age(X_\ultrafilter)$ .

\item   The action  $({\iso}(X))_\ultrafilter\acts {\rm Emb}(Z,X_\ultrafilter)$ is transitive for  every separable $Z\subset X_\ultrafilter$.
%\item   $X_\ultrafilter$ is ultrahomogeneous and  $(\iso(X))_\ultrafilter$ is SOT-dense in $\iso(X_\mc %U)$. 
\item  $X_\ultrafilter$ is Fra\"iss\'e    and $({\iso}(X))_\ultrafilter$ is SOT-dense in ${\iso}(X_\ultrafilter)$   
%For  every $\vep>0$ and  every $X\in \age(E_\ultrafilter)$ there  $\de>0$ such that    $(\iso(E))_\mc %U\acts \Emb_\de(X,E_\ultrafilter)$ is $\vep$-transitive. 
\end{enumerate} 
\end{prop}

\begin{cor}\label{ultra} 
The non-separable $L_p$-space 
$(L_p)_\ultrafilter$ is ultrahomogeneous for each $p\in[1,\infty)$ different from $4,6,8\dots$
\end{cor} 
%\begin{quest} Assume $1 \leq p <+\infty, p \neq 2$. Is there an ultrahomogeneous renorming of $L_p(0,1)$? \end{quest}

%\
\begin{small}

\noindent{\em Comments:}
\begin{itemized}
\item[(1)]
 The proof of Proposition \ref{prop113}
 in \cite{FLMT}  essentially follows the argument given  by Avil\'es, Cabello S\'anchez, Castillo, Gonz\'alez and Moreno in \cite[Section 4.3.3]{ACCGM-LN}. A natural version of Corollary \ref{ultra} for the Gurariy space was also obtained by these authors \cite[Proposition 4.16]{ACCGM-LN}.
 The spaces in Corollary~\ref{ultra} seem to
be the only known super-reflexive ultrahomogeneous examples (if $p>1$).
The existence of {\em separable} ultrahomogeneous spaces other than the Hilbert still remains unknown, {\em cf}. Problem~\ref{prob:multiMazur}.

\item[(2)]
In Corollary~\ref{manytransitive}(a) we observed that when $p\ne 2$, $(L_p)_\ultrafilter$ admits infinitely many  non-isometric transitive renormings. Ho\-wever,
while $(L_p)_\ultrafilter$ is ultrahomogeneous in its natural norm by Corollary~\ref{ultra}, it is not with respect to the norms transferred from   $(L_p(\ell_2^n))_\ultrafilter$ for $n\geq 2$. Indeed these spaces admit both
a $1$-com\-ple\-men\-ted isometric copy of $\ell_2^2$ and another which is not $1$-com\-ple\-men\-ted
(the copy induced by an isometric embedding of $\ell_2^2$ into $L_p$, whose best constant of complementation
is
computed in \cite{Gor69, GLR73}: relapsing into bad habits let us add that it is exactly $\sqrt{{\Gamma(1/2)\Gamma(p/2+1)}/{\Gamma(p/2+1/2)}}$).
%\m{V. I enjoyed that you put the estimate back!} 
 It is not known whether or not there exist  spaces with two or more UH renormings.
\end{itemized}
\end{small}

  \subsection{Local versions of the Fra\"iss\'e property}

If one wants to deduce  some properties of the isometry group
$\Isom(L_p)$ from combinatorial properties
of embeddings between subspaces of $L_p$,
general subspaces of $L_p$ do not seem easy to handle.   Auspiciously, and not unexpectedly, a lot can   be said on the general structure of the space  $L_p$ just from its finite-dimensional $\ell_p$-subspaces. In this direction
we recall  a result of  Schechtman \cite{schechtman} (and Dor \cite{dor} for $p=1$) - as observed by Alspach \cite{alspach}. 
\begin{theorem}[Dor - Schechtman]\label{DS}
Let $1 \leq p<\infty$ be fixed. For every $\eps>0$ there exists $\delta>0$, depending only on $\eps$ and $p$, so that for every $n$ and every  $\delta$-isometry $u:\ell_p^n\to L_p$, there is an isometric embedding $\tilde{u}: \ell_p^n\to L_p$ such that $\|u-\tilde{u}\|<\eps$.  
\end{theorem}

%\begin{theorem}[Dor - Schechtman]\label{DS}
%For each $1 \leq p<\infty$ there is a modulus of stability 
%$\varpi_p:(0,\infty)\to (0,\infty)$ such that    
%$${\Emb}_{\de}(\ell_p^n, L_p(\mu))\subseteq {\Emb}_{\varpi_p(\de)}(\ell_p^n, L_p(\mu)).$$
%for every $n\in \N$, $\de>0$ and every finite measure $\mu$.  
%\end{theorem}

%\m{F: the Th. was stated as: For each $1 \leq p<\infty$ there is a modulus of stability 
%$\varpi_p:(0,\infty)\to (0,\infty)$
%such that    
%${\Emb}_{\de}(\ell_p^n, L_p(\mu))
% \subseteq {\Emb}_{\varpi_p(\de)}(\ell_p^n, L_p(\mu))$
%for every $n\in \N$, $\de>0$ and every finite measure $\mu$. 
%\
%Of course this does not mean anything. Then I went to [FLMT] and realized that we should have written ${\Emb}(\ell_p^n, L_p(\mu))_{\varpi_p(\de)}$ instead of $ {\Emb}_{\varpi_p(\de)}(\ell_p^n, L_p(\mu))$. But as the former expression is not defined here nowadays (and neither is ``modulus of stability'') I decided to avoid both (and also the measure $\mu$ which plays no role here and actually is a bit misleading in view of the subsequent comment) and state the Th. in its present form. Please check and, eventually, erase this note.
%  }
In other words, a form of the Fra\"iss\'e property in $L_p$ is satisfied  when ``restricted" to  subspaces of $L_p$ isometric to an $\ell_p^n$
(including $p=4,6,8\dots$). As commented earlier there is no hope to extend this to general finite dimensional subspaces when $p=4,6,8\ldots$

With this example in mind, it is possible and useful  to develop a Fra\"iss\'e  theory with respect to restricted classes
of finite dimensional subspaces, which are not the Age of any $X$, because they are not hereditary.
In this sense this may be called a ``local version" of the Fra\"iss\'e theory for Banach spaces. 
Informally, given a class $\mathscr F$ of finite dimensional Banach spaces, the
$\mathscr F$-Fra\"iss\'e spaces are those for which the natural actions on $\de$-embeddings are $\vep$-transitive,
provided that the embeddings have as domain an element of $\mathscr F$. 
For $L_p$ the authors of \cite{FLMT} 
use  Theorem \ref{DS} to give meaning to the affirmation:

\begin{thm} For any $p\in[1,\infty)$, even or not, $L_p$ is the  Fra\"iss\'e limit of the class 
  $(\ell_p^n)_n$.  
\end{thm}

\subsection{Fra\"iss\'e and Extreme Amenability}

Recall from Section~\ref{sec:not+con} %\m{F: Valentin, please append the label $\{\text{sec:not+con}\}$ (subsection) of notations and conventions}
that a topological group $G$ is called {extremely amenable (EA)} when every continuous action 
$G\curvearrowright K$ on a compact $K$ has a fixed point. 
The Fra\"iss\'e theory is related to this notion through  the celebrated 
\  {\em KPT correspondence}\ (Kechris/Pestov/Todorcevic    2005 \cite{KPT}), a combinatorial characterization 
of the extreme amenability of an automorphism group in terms of a Ramsey property of   Age: as a beautiful example,  Pestov's result that the group of order preserving automorphisms of the rationals
 is extremely amenable \cite{Pe2} may be seen as  combination of
``$(\Q,<)$ is the Fra\"iss\'e limit of finite ordered sets" and of the classical finite Ramsey theorem on $\N$. The authors of \cite{FLMT} use
a form of the KPT correspondence for metric structures which
 applies without difficulty to the isometry group of a Fra\"iss\'e, or even \auh,\ Banach space $X$.

 \begin{definition}
A collection  $\mathscr F$ of finite dimensional normed spaces has the {Approximate Ramsey Property (ARP)} when
for every $F,G \in \mathscr F$ and  $r \in \N, {\vep}>0$   there exists $H\in \mathscr F$ such that    every coloring $c$ 
of $\mathrm{Emb}(F,H)$ into $r$ colors admits an embedding $\ro \in \mathrm{Emb}(G,H)$ which is  $\vep$-monochro\-ma\-tic for $c$: there exists a color $i$ so that for all $u\in {\Emb}(F,G) $ there is $v\in \mathrm{Emb}(F,H)$ such that $c(v)=i$ and $\|v-\ro u\|\leq \vep$. 
\end{definition}
 %Here {  $\vep$-monochromatic} means that  for some color $i$,
 %$\mr{osc}(c\rest \ro\circ \mr{Emb}(F,G))<\vep.$
 %$\ro\circ \mr{Emb}(F,G)$ is contained in the $\vep$-expansion of $c^{-1}(i)$.
 %\vep:=\{\tau \in \mr{Emb}(F,H): d(c^{-1}(i),\tau)<\vep\}$.

\begin{thm} [KPT correspondence for  Banach spaces]\label{kptc}
For  an   \auh\ Banach space $X$ the following are equivalent:
\begin{enumerate}
\item $\mr{Isom}(X)$ is extremely amenable for SOT.
\item  $\age(X)$ has the approximate Ramsey property. 
\end{enumerate}
\end{thm}
The KPT correspondence turns out to extend to the setting of $\ell_p^n$-subspa\-ces of $L_p$.
This means that one can expect to prove the extreme amenability of ${\rm Isom}(L_p)$ through {internal} properties, 
i.e. through an approximate Ramsey property of the class of isometric   embeddings between $\ell_p^n$'s. This expectation
was fulfilled for $p=\infty$ in \cite{BaLALuMbo2}, and then for $1 \leq p <\infty, p \neq 2$, in   \cite{FLMT}:

\begin{thm}[Ramsey theorem for  embeddings be\-tween $\ell_p^n$'s]\label{ram}
\

 Given $1\leq p\le\infty, p \neq 2$, integers $d$, $m$, $r$, and $\vep>0$ there exists
 ${ n}=n_p({ d,m},r,\vep)$ such that whenever  $c$ is a coloring of ${\rm Emb}(\ell_p^d,\ell_p^n)$ into $r$ co\-lors,
 there exists some  isometric embedding $\gamma:\ell_p^m\to \ell_p^n$ 
%such that 
%$
%\gamma \circ {\rm Emb}(\ell_p^d,\ell_p^m)
%\subset (c^{-1}\{i\})_\epsilon.
%$
which is $\vep$-monochromatic.
\end{thm}

The proofs of these two results are quite complex and beyond the scope of this survey.
The proof obtained in \cite{FLMT} for $p \neq 2$, as well as the estimates on $n_p(d,m,r,\vep)$ that would follow from it (but are not computed by the authors), do not extend to the case of the Hilbert, due to the different nature  of isometric embeddings between finite dimensional subspaces in this case. Theorem \ref{ram} is still valid for $p=2$, but as a consequence of
Theorem \ref{kptc} and of Gromov-Milman's result \cite{GM} of extreme amenability of the unitary group, see the following comments.

The Fra\"iss\'e spaces encountered in \cite{BaLALuMbo1,BaLALuMbo2, FLMT} were known to have extremely amenable isometry groups when equipped with the strong operator topology:

\begin{ex} The isometry group of $L_p$, for any $1 \leq p<\infty, p \neq 2$, and the isometry group of ${\mathcal G}$ are extremely amenable in the SOT. 
 \end{ex}

The extreme amenability of ${\rm Isom}(L_p)$ for $p\in [1,\infty),  p \neq 2$ was proved in 2006 by Giordano and Pestov \cite{GiPe},
and the methods of \cite{FLMT} allow to recover this result    through Theorem \ref{ram}  and the 
KPT correspondence for Banach spaces. In any case this statement refers to one group only because Choksi and  Kakutani proved long time ago that the groups $\Isom(L_p)$ are all topologically isomorphic in the SOT for $p\in[1,\infty)$, see \cite[Theorem~8]{kakutani}.

The extreme amenability of  ${\rm Isom}(\mathcal G)$ is a  recent result by Barto\-so\-v\'a, L\'o\-pez-Abad, Lupini, and Mbombo \cite{BaLALuMbo1, BaLALuMbo2}, and may be seen as a corollary of the combination of the KPT correspondence and Theorem \ref{ram} for $p=\infty$.

When $p=2$ the isometry group of $L_p$ is the unitary group whose extreme amenability was established in 1983 by Gromov and Milman \cite{GrMi}.  

The KPT correspondence for Banach spaces also implies new results for some non-separable versions of those spaces. As a consequence of Gromov-Milman's result and of Theorem \ref{kptc}, the unitary group of any infinite dimensional Hilbert space is extremely amenable for the SOT, regardless of its density character. From Theorem \ref{kptc}, the result that $L_p, 1 \leq p <\infty$ is Fra\"iss\'e, and Proposition \ref{prop113}, we also have:

\begin{ex} 
For $1 \leq p <\infty, p \neq 2$ the isometry group of any ultrapower of $L_p$ with respect to a free ultrafilter on the integers is extremely amenable in the SOT.
\end{ex}

\begin{small}
\noindent{\em Comments:}
\begin{itemized}
  \item[(1)]
  The Gromov-Milman \cite{GrMi} result of extreme amenability of ${\ku U}({\ku H})$ is ba\-sed on the concentration of measure phenomenon.
 The result of Giordano-Pestov \cite{GiPe} for $L_p$ also uses concentration of measure and a general description of ${\rm Isom}(L_p)$ as a 
 topological group. In \cite{FLMT} and for $p$ not even this may be recovered by the above considerations through the fact that $L_p$ is
 \auh\ and through the Ramsey property; for $p$ even,  the local version of the Ramsey property,  
  Theorem~\ref{ram}, needs to be used. 
 The extreme amenability of the isometry group of the Gurariy relies on its
 \auh  \ property and
 the Ramsey property for embeddings between finite dimensional spaces (or equivalently, between $\ell_\infty^n$-spaces).

\item[(2)] There are some precursors of the Ramsey result Theorem \ref{ram}. In \cite{ORS},
  Odell-Rosenthal-Schlumprecht (1993) proved that that for every $1\le p\le \infty$, 
 every $m, r \in \N$ and every $\vep>0$ there is $n \in \N$ such that for every coloring $c$ 
 of $S_{\ell_p^n}$ into $r$ there is $Y\subset \ell_p^n$ isometric to $\ell_p^m$ so that $S_Y$ is $\eps$-monochromatic. 
 Note that Odell-Rosenthal-Schlumprecht is the case $d=1$ in  Theorem \ref{ram}.
Matou\v{s}ek-R\"odl (1995) \cite{MR} gave a combinatorial proof of the \cite{ORS} result for $1\le p<\infty$.

\item[(3)]
 The authors of \cite{FLMT} also develop a Fra\"iss\'e theory by restricting the type of embeddings, 
 for example by analyzing  lattices,
 where now isometries and embeddings (resp. $\delta$-embeddings) must respect (resp. maybe up to $\eps$) the lattice structure. 
In this manner {Fra\"iss\'e Banach lattices}, 
i.e. certain unique universal objects for classes of finite dimensional lattices with an approximate lattice ultrahomogeneity property, are defined \cite[Definition 6.1]{FLMT}.

For example, $L_p$ is a Fra\"iss\'e Banach lattice for $p\in[1,\infty)$
which is the ``lattice Fra\"iss\'e" limit of its finitely generated   sublattices the $\ell_p^n$'s.
For $p=\infty$ they define a new object which they call the ``Gurariy $M$-space'', proving that
there exists a renorming of $C[0,1]$ as an $M$-space which is the lattice Fra\"iss\'e limit of
the class of $\ell_\infty^n$'s finite lattices.
\item[(4)]
The ``Gurariy $M$-space" cited in the previous item is inspired from a couple of constructions from \cite{CS7}; namely an
$M$-space, which is transitive and easier to define, albeit non-separable (the ultraproduct of the lattices $L_p$ for $p$ tending to $\infty$) and a family of separable AT $M$-spaces some of which (all?) might be isometric to the ``Gurariy $M$-space'' ... or not. 
Avil\'es and Tradacete \cite{AvTr} and M.A. Tursi \cite{Tursi} recently and independently
studied amalgamation properties for Banach lattices: Avil\'es and Tradacete constructed a (necessarily non-separable) Banach lattice of universal disposition for separable lattices.
Tursi's paper contains, among other things, the construction of a separable approximately ultrahomogeneous Banach lattice. 
In the even more recent \cite{KL}, Kawach and L\'opez-Abad study amalgamation and Fra\"iss\'e properties for Fr\'echet spaces.
\end{itemized}
\end{small}

 \section{Questions and problems}
In this final section we gather and discuss a number of problems that arise naturally  
 from the contents of the survey. %Some of them might be easy or interesting. 
 We have classified them according to the topics covered in the preceding sections, although the borders are quite permeable. 
 
 \subsection*{Local questions, ultraproducts, finite dimensional objects}
It is clear (use 
the $\sqrt{\dim}$ estimate for the ellipsoid of minimal volume or
an ultraproduct argument) that for each finite $n$ there is a function $f_n:[0,2]\to[1,\infty)$ with $f_n(\delta)\to 1$ as $\delta\to 0$ so that if $E$ is $n$-dimensional and $\delta$-transitive, then $d_{\text{BM}}(E,\ell_2^n)\leq f_n(\delta)$. See the comments closing Section~\ref{sec:classical-isom}.

\begin{problemo}
Can the hypothesis on the dimension  be removed? That is, is it true that for every $\eps>0$ there exist $\delta>0$ so that every finite dimensional $\delta$-transitive (or $\delta$-asymptotically transitive) space is $(1+\eps)$-close to the Hilbert space of the corresponding dimension?
\end{problemo} 

An obvious ultraproduct argument in combination with Theorem~\ref{hyperplane} shows that the answer is affirmative if we moreover require the existence of a $(1+\delta)$-complemented hyperplane.

Banach spaces that arise as ultraproducts of families of finite dimensional ones are called {\em hyperfinite} in nonstandard ambients, see for instance \cite{h-m}. A couple of closely related question are:

\begin{problemo}
Is every hyperfinite transitive (or even ultrahomogeneous) space  (isometric or isomorphic to) a Hilbert space?
\end{problemo}

\begin{problemo}(Henson and Moore \cite[Problem 5]{h-m}, Plichko)
\hspace{2pt}Do hyperfinite spaces of universal disposition exist?
\end{problemo}

\begin{problemo}\label{pr:Fdimiso} (F. Cabello S\'anchez)\hspace{2pt}
 Let $X$ be a space with an (almost) transitive norm and which admits a non trivial finite-dimensional isometry. Must $X$ be hilbertian?
 \end{problemo}
 
 The answer is affirmative if the hypothesized finite-dimensional isometry is a rank-one perturbation of the identity (see \cite[Section 3]{becerra}).
Also, by \cite[Corollary 4.14]{FR}, if $X$ is separable reflexive and satisfies the hypotheses in Problem~\ref{pr:Fdimiso}, then $X$ must have a Schauder basis. 
 With an eye in Theorem~\ref{hyperplane} we can ask:

\begin{problemo} (F. Cabello S\'anchez)\hspace{2pt}
 Let $X$ be a space with an (almost) transitive norm and which admits a 1-complemented subspace of finite codimension greater than 1. Must $X$ be hilbertian?
 \end{problemo}
 
 Let $\GL_f(X)$ denote the group of automorphisms of $X$ that have the form ${\bf I}_X+F$ where $F$ is a finite-rank operator.
 In   \cite[Problem 8.11]{FR}  asked to find a separable space $X$ and a bounded subgroup of $\GL_f(X)$ which is infinite and discrete in the SOT. This was solved in \cite{AFGR} with an example on $c_0$. The question remains in reflexive spaces:
 
 \begin{problemo} Find a separable reflexive space $X$ and a bounded subgroup of $\GL_f(X)$ which is infinite and discrete for the SOT.
 \end{problemo}
 
 In the same vein we ask:
 
 \begin{problemo} If $X$ is separable reflexive and a bounded
 subgroup $G$  of $\GL_f(X)$ is discrete for the SOT, does it imply that all orbits of the action of $G$ on the sphere are discrete? Or at least, not dense in $S_X$?
 \end{problemo}

 \subsection*{Maximality of the norm, renormings}
Not surprisingly, the hottest issue in this line is about norms on Hilbert spaces: 

\begin{problemo}(Section~\ref{sec:smAT}, Problem~\ref{3problems})\
%\m{B: possibly we need some comments as above, Note to self: still to do}
\begin{itemize}
\item[(a)]  Does the Hilbert space have a  unique, up to conjugacy,  maximal bounded subgroup of automorphisms?
\item[(b)]  Does there exist a separable Banach space $X$  with a unique, up to conjugacy,  maximal  bounded subgroup of $\GL(X)$? 

\item[(c)]  If yes, does $X$ have to be 
 isomorphic to a Hilbert space?
 \end{itemize}
\end{problemo}
 
Note that, while the isometric part of Mazur problem asks whether every (almost) transitive renorming of a Hilbert space is Euclidean, Part (a) is asking if this is true even for maximal renormings. Concerning the possible impact that the existence of AT norms can have regarding the isomorphic structure of the underlying space:

 \begin{problemo}(\cite[Problem 8.14]{FR})\hspace{2pt}
 Let $X$ be a separable, reflexive, Banach space with an AT norm. Does it follow that $X$ has a Schauder basis?
 \end{problemo}
 
This was originally asked for CT norms. However, as we already mentioned, CT and AT are equivalent notions for reflexive spaces and imply uniform convexity and uniform smoothness of the norm;  cf. \cite[Corollary 6.9]{becerra}. Removing the hypothesis of reflexivity the answer is no in view of Lusky's \cite{lusky-rot}.
 By \cite[Corollary 4.14]{FR}, the answer is affirmative when there exists a power bounded operator in $\GL_f(X)$.

\begin{problemo}
 Assume that $X$ is a (complex) HI space. Show that $X$ does no admit an almost transitive norm, or even,  that the isometry group acts almost trivially on $X$.
\end{problemo}

According to \cite[Corollary 6.7]{FR} the answer to this problem is affirmative  when $X$ is a separable reflexive HI space without a Schauder basis.

All the examples appearing in
Theorem~\ref{list} are, in some sense, ``far from being Hilbert''. One may wonder if there exist counterexamples within the most popular classes of spaces that are ``close to being Hilbert'':

\begin{problemo}
Find a  weak Hilbert space, an asymptotically hilbertian space, or even a near Hilbert space that does not admit an AT renorming.
 \end{problemo}

Please note that {\em asymptotically hilbertian} is not the same as  asymptotic (or Asymptotic) $\ell_p$ space for $p=2$; see Theorem~\ref{list}. The definition of weak-Hilbert and asymptotically hilbertian spaces, as well as various characterizations, can be seen in Pisier~\cite{pisier-weak}; a near Hilbert space is just a Banach space having type $2-\eps$ and cotype $2+\eps $ for every $\eps>0$.
These include all ``twisted Hilbert spaces'', in particular the Kalton-Peck spaces \cite{kp}. Going in the opposite direction (see the comments closing Section~\ref{subDGZ}):

\begin{problemo}
%Does the space $S(T^{(2)})$ admit an almost transitive renorming? 
Does there exist any symmetric space not isomorphic to $\ell_2$ which admits an almost transitive renorming? \end{problemo}

\begin{problemo}
 Are the Hilbert spaces the only (almost) transitive spaces with property $(M)$? %\m{Added question about Property $(M)$}
 %\m{I suggest to separate this into its own problem, since this is quite differen from the preceding problem. But I will agree to whateve the two of you decide.}
 \end{problemo}

\begin{problemo}
%\m{F: This includes the question mentioned by Beata in the note on p. 26. Please check and, if everything is ok, close that note and also this one}
Does there exist a separable Banach space so that every bounded subgroup of $\GL(X)$ is contained in some maximal bounded subgroup of $\GL(X)$?
Is this true for $L_p$ or Kadec' complementably universal space?% $1\le p <\infty$, $p\ne 2$?
\end{problemo}

In view of Theorem~\ref{main2} this problem could have different answers for $L_p$ and $\mathcal K$ since the latter contains a complemented copy of each separable HI space with the BAP.

 \begin{problemo}
 Does $T^{(2)}$, the 2-con\-ve\-xified Tsirelson space,  %\m{B: include this also in section 5 on open problems}
or do more general weak Hilbert spaces, other than the Hilbert, have a  maximal  bounded subgroup of $\GL(X)$?
\end{problemo}

\begin{problemo}\lb{pr:allsubspbis}
  (Dilworth and Randrianantoanina~\cite[Problem 1.1]{DR})\hspace{2pt} Suppose that every subspace of a Banach space $X$ admits  an equivalent almost transitive renorming. Is $X$ isomorphic to  a Hilbert space?
\end{problemo}

Going back to the genuine Mazur affairs we find the following question, especially the case $p=1$, most itching:
 
\begin{problemo}\label{pr:LpT}
Does $L_p$ admit a {\em transitive} renorming for some $p\neq 2$? %\in[1,\infty)$ 
%different from $2$?
 \end{problemo}  
%\begin{problem}

%\end{problem}

%Since by \cite{FR2} such a norm would have to be strictly convex, maybe  the case $p=1$ is not out of hand.\m{should we keep this comment?}

%We do not know whether there exists a separable Banach space with two or more 
%non-conjugate  equivalent transitive norms, see  Problems 3.1 and 3.2 (labels needed) and the discussion in Section~\ref{?} 3.1.

%{\tt Maybe we should state it as a problem.}

%{\tt Also, were there known other spaces satisfying 

%Corollary~\ref{manytransitive}(2.23)? I think that %Felix said yes, but I forgot the details}

\subsection*{Problems relative to Fra\"iss\'e or homogeneous spaces}
Here, the fundamental question seems to be Problem~\ref{prob:multiMazur}, namely

\begin{problemo}(Multidimensional Mazur problem) \hspace{2pt} %\label{prob:multiMazur} 
Is every separable %in\-fi\-ni\-te dimensional 
ultrahomogeneous Banach space isometric (or isomorphic) to the Hilbert space?
\end{problemo}

Even in this setting the gap between an eventual affirmative answer and the existing knowledge is sideral.

\begin{problemo}\label{pr:GLptheonlyAuH}(\cite[Problem 2.9]{FLMT})\hspace{2pt} Are the Gurariy space and the spaces $L_p$ for $ p \neq 4,6,8,\ldots$ the only separable Fra\"iss\'e spaces? or even AUH spaces?
\end{problemo}

Variants of this problem were suggested to us by G. Godefroy, based on the
 well-known fact that the norm on $L_p$ is a $C^\infty$-smooth norm exactly when $p$ is even
 %\m{More details given here, following the remark of the referee about p35 of the original version}
 (see \cite[Chapter~V]{DGZ} for much more information on this, in particular for a proof that the canonical norm of $L_p$ is optimal regarding smoothness for $1<p<\infty$). For example:
 
 \begin{problemo} Show that the Hilbert space is the only separable Fra\"iss\'e (or even AUH) space with a $C^\infty$-smooth norm.
 \end{problemo}
 
 \begin{problemo} Show that a $C^\infty$-smooth norm which is Fra\"iss\'e (or even AUH) is necessarily ultrahomogeneous.
 \end{problemo}

%\begin{problem} Is every (separable) AUH space necessarily Fra\"iss\'e?
%\end{problem}

Note that any Fra\"iss\'e renorming of the Gurariy space must be isometric to the Gurariy itself. Indeed, cotype considerations imply that $\ell_\infty$ is finitely representable in such space, and then  we may apply the observation after Proposition \ref{unique}.  The question seems to remain open for $L_p$:

\begin{problemo} Let $1\le p <\infty$.  Is any Fra\"iss\'e norm on $L_p$ conjugate to the usual norm? %necessarily isometric to the Hilbert space? Same question for 
%$L_p(0,1), 1 \le p <+\infty$?
\end{problemo}

The multidimensional version of Problem~\ref{pr:LpT} is:

\begin{problemo}\label{pr:LpUH} Show that $L_p$ does not admit an ultrahomogeneous renorming.% for $p \neq 2,4,6,\ldots$
\end{problemo}

\begin{problemo} (\cite[Problem 2.6]{FLMT})\hspace{2pt} Are the Gurariy space and the Hilbert space the only separable stable Fra\"iss\'e Banach spaces?
\end{problemo}

In particular,

\begin{problemo} (Problem~\ref{pr:LpstableFraisse}) \hspace{2pt}
Are the spaces $L_p$, $p \neq 2,4,6,\ldots$ stable Fra\"iss\'e?
\end{problemo}

In relation to   \cite[Proposition 2.14]{FLMT} we may ask:

\begin{problemo} Is every (separable) AUH space necessarily Fra\"iss\'e? Is every ultrahomogeneous space  Fra\"iss\'e?
Is every space having an ultrahomogeneous (``countable'') ultrapower  Fra\"iss\'e?
\end{problemo}

%\begin{problemo} Does there exist a Banach space which is not Fra\"iss\'e but having an ultrahomogeneous (``countable'') ultrapower?
%\end{problemo}

\subsection*{The Banach-Gromov ``conjecture''} %\m{Added paragraphs about Banach-Gromov}
%Following a suggestion of the referee, 
We close with a few remarks on another problem in Banach's book (remarques au Chapitre XII, p. 215) concerning isometric characterizations of Hilbert spaces. %We consider only {\em real} spaces in the ensuing discussion.

\begin{problemo}\label{BG}(Banach)\hspace{2pt} Let $X$ be a Banach space such that, for some $2\le n<\infty$, all $n$-dimen\-sional subspaces of $X$ are isometric. Must $X$ be a Hilbert space?
\end{problemo} 

The hypothesis on $X$ in this problem is somehow ``dual'' 
to that appearing in Mazur's: note that all hyperplanes in a reflexive  transitive space are mutually isometric.

 An affirmative answer for $n=2$ (real case)  was quickly provided by Auerbach, Mazur, and Ulam in \cite{AMU} and then for infinite-dimensional $X$ and any $n$ by Dvoretzky \cite{dvo} (the complex version of Dvoretzky's theorem was established later by Milman \cite{milman}) making it clear that Banach's question reduced to considering hyperplanes in finite-dimensional spaces.
%\begin{problemo} %
%Let $X$ be an $n+1$-dimensional space all whose hyperplanes are isometric. Must $X$ be a Hilbert space?
%\end{problemo}
In 1967, Gromov \cite{gromov} solved the problem in the affirmative for even $n$ and all $X$ (real or complex),
for odd $n$ and real $X$ with $\dim(X) \geq  n + 2$, and for odd $n$ and complex $X$ with $\dim(X) \ge 2n$, which also settled the problem in  any infinite dimensional $X$.
Thus, the first integer for which Banach's problem remains open is $3$.
 We refer ther reader to Soltan \cite[Section 6]{soltan} for more information on this topic and to \cite{Bor} for a  recent result. 
\bigskip

\noindent{\em Acknowledgements:} 
The authors started discussing questions related to this survey  in 2015, at the conference ``Banach spaces and their applications
in analysis'' at the Centre International de Rencontres Math\'ematiques (CIRM) in Luminy, France, and would like to thank the organizers for 
their hospitality and support. We thank G. Godefroy for helpful comments on a previous version of this survey. %Finally we thank the referee for the careful reading of this survey and for helpful suggestions.
%We  also would like to thank our various coauthors of the works cited in this survey. \m{V. added this about coauthors and SP Journal, feel free to correct or improve!} Finally the authors thank the editors of the S\~ao Paulo Journal of Mathematical Sciences for an opportunity \m{B: were you invited to write this survey?} to publish this survey.

%\section{Declarations}

%{\bf Funding}: The research of the first author was supported in part by Pro\-jects MICINN MTM2016-76958-
%C2-1-P and PID2019-103961GB-C21, and Project IB16056, Junta de Extremadura.
%The research of the second author was supported by FAPESP, grants 2016/25574-8 and by CNPq, grant 303731/2019-2. 

%{\bf Conflicts of interest/Competing interests}: Not applicable. 

%{\bf Availability of data and material}: Not applicable.

%{\bf Code availability}: Not applicable.

%{\bf Authors' contributions}: Not applicable. 

\end{document}